\documentclass[11pt]{amsart}

\usepackage{amsfonts, amsmath, amscd}
\usepackage[psamsfonts]{amssymb}

\usepackage{amssymb}

\usepackage{pb-diagram}

\newcommand{\uuu}{\mathbf{u}}
\newcommand{\NN}{\mathbb{N}}

\newcommand{\TT}{\mathbb{T}}
\newcommand{\ZZ}{\mathbb{Z}}
\newcommand{\xxx}{\mathbf{x}}

\newcommand{\SP}{\mathbf{sp}}

\newtheorem{theorem}{Theorem}
\newtheorem{lemma}[theorem]{Lemma}
\newtheorem{remark}[theorem]{Remark}
\newtheorem{proposition}[theorem]{Proposition}
\newtheorem{definition}[theorem]{Definition}
\newtheorem{corollary}[theorem]{Corollary}
\newtheorem{problem}[theorem]{Problem}

\newtheorem{fact}[theorem]{Fact}
\newtheorem{example}[theorem]{Example}

\numberwithin{equation}{section}
\numberwithin{theorem}{section}

\begin{document}


\title[On topological properties of  the group of the null sequences]{On topological properties of  the group of the null sequences valued in an Abelian topological group}

\author{S.S. Gabriyelyan}
\email{saak@math.bgu.ac.il}

\address{Department of Mathematics, Ben-Gurion University of the
Negev, Beer-Sheva, P.O. 653, Israel}


\begin{abstract}
Following \cite{DMPT}, denote by  $\mathfrak{F}_0$ the functor on the category $\mathbf{TAG}$ of all Hausdorff Abelian topological groups and continuous homomorphisms  which passes each $X\in \mathbf{TAG}$ to the group of all $X$-valued null sequences endowed with  the uniform topology. We prove that if $X\in \mathbf{TAG}$ is an $(E)$-space (respectively, a strictly angelic space or a \v{S}-space), then $\mathfrak{F}_0 (X)$  is an $(E)$-space (respectively, a strictly angelic space or a \v{S}-space).  We study respected properties for topological groups in particular from categorical point of view. Using this investigation we show that for a locally compact Abelian (LCA) group $X$ the following are equivalent: 1) $X$ is totally disconnected, 2) $\mathfrak{F}_0 (X)$ is a Schwartz group, 3) $\mathfrak{F}_0 (X)$ respects compactness, 4) $\mathfrak{F}_0(X)$ has the Schur property. So, if a LCA group $X$ has non-zero connected component, the group $\mathfrak{F}_0(X)$ is a reflexive non-Schwartz group which does not have the Schur property. We prove also that  for every  compact connected metrizable Abelian group $X$ the group $\mathfrak{F}_0 (X)$ is  monothetic that generalizes a result by Rolewicz for $X=\TT$. 
\end{abstract}






\maketitle

\section{Introduction}

{\bf Properties and functors in $\mathbf{TAG}$}.
Denote by  $\mathbf{TAG}$  the category   of all Hausdorff Abelian topological groups  and continuous homomorphisms.
Below we consider some important subcategories, functors and properties in $\mathbf{TAG}$ which are essentially used in the article.


(A) {\it Properties in} $\mathbf{TAG}$. We consider the following three types  of properties in  $\mathbf{TAG}$. The properties of the type $\mathcal{P}_s$ are those ones which are preserved under taking closed subgroups (i.e., if $X\in \mathbf{TAG}$ has a property $\mathcal{P}$ and $H$ is a closed subgroup of $X$, then $H$ also has $\mathcal{P}$): (pre)compactness, local (pre)compactness, realcompactness, sequential compactness, countable compactness, completeness, sequential completeness, normality,  sequentiality, Fr\'{e}chet-Urycohnness, metrizability etc. \cite{ArT, DikTka, Eng, GiJ}. The properties which are preserved under taking quotients form the type  $\mathcal{P}_q$:  connectedness and pseudocompactnes etc. \cite{ArT, Eng, GiJ}. We denote by $\mathcal{P}_n$ the properties which are not preserved under taking neither closed subgroups nor quotients:  (Pontryagin) reflexivity etc. (see, the survey \cite{CDM}).

Let $\pmb{\mathbb{A}}$ be a subcategory of $\mathbf{TAG}$ and let  $F: \mathbf{TAG} \to \mathbf{TAG}$ be a functor. We say that $\pmb{\mathbb{A}}$ is {\it $F$-invariant} if $F(X)\in \pmb{\mathbb{A}}$ for every $X\in \pmb{\mathbb{A}}$.
We say that the functor $F$ {\it preserves a property $\mathcal{P}$ on $\pmb{\mathbb{A}}$} if $F(X)\in \mathbf{TAG}\cap \mathcal{P}$ for each $X\in \pmb{\mathbb{A}} \cap \mathcal{P}$.  The next two general problems are natural:
\begin{problem} \label{problem1}
Characterize all groups $X\in \pmb{\mathbb{A}}\cap \mathcal{P}$ such that $F(X)\in \mathbf{TAG}\cap \mathcal{P}$. Does the functor $F$ preserve the property $\mathcal{P}$ on $\pmb{\mathbb{A}}$?
\end{problem}

\begin{problem} \label{problem2}
Let $F(X) \in \pmb{\mathbb{A}}\cap \mathcal{P}$ for $X\in \pmb{\mathbb{A}}$. Does also $X\in \pmb{\mathbb{A}} \cap\mathcal{P}$?
\end{problem}


(B) {\it The Bohr functor}. For an Abelian topological group  $X$ we   denoted by $\widehat{X}$ the group of all continuous characters on $X$.  The group $X$ is called {\it maximally almost periodic} (MAP) if $\widehat{X}$ separates the points of $X$. Denote by $\mathbf{MAPA}$ the full subcategory of  $\mathbf{TAG}$ consisting of all MAP Abelian groups. The class $\mathbf{LCA}$ of all  locally compact Abelian (LCA for short) groups is the most important full subcategory of $\mathbf{MAPA}$.
For    $(X,\tau)\in \mathbf{MAPA}$ we denote by $\sigma(X,\widehat{X})$ or $\tau^+$ the {\it weak topology} on $X$, i.e., the smallest group topology in $X$ for which the elements of $\widehat{X}$ are continuous.
The topology $\tau^+$ is called also {\it the Bohr modification} of $\tau$. Denote by $\mathfrak{B}: \mathbf{MAPA} \to \mathbf{MAPA}$ the Bohr functor, i.e., $\mathfrak{B}(X,\tau):= X^+$ for every $(X,\tau)\in \mathbf{MAPA}$, where $X^+ :=(X,\tau^+)$. Note that $\mathfrak{B}$ is finitely multiplicative. It is well-known that the groups $X$ and $X^+$ have the same set of continuous characters, and  $\mathfrak{B}(X)=X$ if and only if $X$ is precompact (see \cite{ArT, CoR}).

It is known that, on the class $\mathbf{LCA}$ the Bohr functor $\mathfrak{B}$ preserves connectedness \cite{Tri91}, covering dimension \cite{Her98, Tri91} and realcompactness \cite{CHTr}. On the other hand, $X^+$ is normal if and only if $X$ is $\sigma$-compact \cite{Tri94}. Note also that, if $X$ is discrete  and infinite, then $X^+$ is sequentially complete \cite{DikTka} and non-pseudocompact \cite{CS73}. The same holds for the general case: if $X$ is a non-compact LCA group, then  $X^+$ is sequentially complete and non-pseudocompact \cite{Dikr}.


(C) {\it Respected properties}.
Using the Bohr functor we can define {\it respected properties} in subcategories of $\mathbf{MAPA}$.
Following \cite{ReT3}, if $\mathcal{P}$ denotes a topological property, then we say that a $MAPA$ group $X$ {\it respects} $\mathcal{P}$ if $X$ and $X^+$ have the same subspaces with $\mathcal{P}$.
Let $\pmb{\mathbb{A}}$ be a subcategory of  $\mathbf{MAPA}$.
\begin{problem} \label{problem3}
Characterize  $X \in \pmb{\mathbb{A}}$ which respect $\mathcal{P}$.
\end{problem}
We say that $\mathcal{P}$ is a {\it respected property} in  $\pmb{\mathbb{A}}$ if every $X\in \pmb{\mathbb{A}}$  respects $\mathcal{P}$.
\begin{problem} \label{problem30}
Which topological properties are respected for a given subcategory $\pmb{\mathbb{A}}$ of $\mathbf{MAPA}$?
\end{problem}

Recall that $X\in \mathbf{MAPA}$ is said to have: 1) the {\it  Glicksberg property} or $X$  {\it respects compactness} if the compact subsets of $X$ and $X^+$ coincide, and 2) the {\it Schur property} or $X$  {\it respects sequentiality} if $X$ and $X^+$ have the same set of convergent sequences. Clearly, if $X$ respects compactness, then it   has the Schur property.

Let $X\in \mathbf{LCA}$ and let $\mathcal{P}$ be a topological property. The question of whether $X$ respects $\mathcal{P}$ has been intensively studied for many natural properties $\mathcal{P}$. The famous Glicksberg theorem \cite{Gli} states that every LCA group $X$ respects compactness, and hence $X$ has the Schur property. So the  Glicksberg and the Schur properties are respected properties in $\mathbf{LCA}$. Trigos-Arrieta \cite{TriAr91, Tri91} proved that also  pseudocompactness  and functional boundedness are respected properties in $\mathbf{LCA}$. Banaszczyk and Mart\'{\i}n-Peinador \cite{BaMP2} generalized these results to the class $\mathbf{Nuc}$ of all nuclear groups. Hern\'{a}ndez, Galindo and Macario \cite{HGM} characterized Banach spaces which have the Schur property. Recently Au{\ss}enhofer \cite{Aus2} proved that every locally quasi-convex Schwartz group respects compactness.
These results lead us to the next question which is also considered in the article:
\begin{problem} \label{problem4}
Let $F$ be a functor in $\mathbf{MAPA}$ and let $\mathcal{P}$ be a respected property in a subcategory $\pmb{\mathbb{A}}$ of  $\mathbf{MAPA}$. Characterize  $X\in \pmb{\mathbb{A}}$ such that $F(X)$ respects $\mathcal{P}$.
\end{problem}


(D) {\it The functor $\mathfrak{F}_0$ in $\mathbf{TAG}$}.
Now we define the functor $\mathfrak{F}_0$ in $\mathbf{TAG}$ introduced by  Dikranjan, Mart\'{\i}n-Peinador and  Tarieladze in \cite{DMPT}.
Let $X$ be an Abelian topological group and let $\mathcal{N}(X)$ be the filter of all open neighborhoods at zero in $X$. Denote by $X^\mathbb{N}$ the group of all sequences $\xxx =(x_n)_{n\in\NN}$. The subgroup of $X^\mathbb{N}$ of all sequences eventually equal to zero we denote by $X^{(\mathbb{N})}$.
It is easy to check  that the collection $\{ V^\mathbb{N}: V \in \mathcal{N}(X)\}$ forms a base at $0$ for a group topology in $X^\mathbb{N}$. This topology  is called {\it the uniform topology} and is denoted by $\mathfrak{u}$ \cite{DMPT}. Following \cite{DMPT}, denote by $c_0 (X)$ the following subgroup of $X^\mathbb{N}$
\[
c_0 (X) := \left\{ (x_n)_{n\in\mathbb{N}} \in X^\mathbb{N} : \; \lim_{n} x_n = 0 \right\}.
\]
The uniform group topology on $c_0 (X)$ induced from $(X^\mathbb{N}, \mathfrak{u})$ we denote by  $\mathfrak{u}_0$. Following \cite{DMPT}, define the functor $\mathfrak{F}_0$ on  the category $\mathbf{TAG}$ by the assignment
\[
X\to \mathfrak{F}_0 (X) :=(c_0 (X), \mathfrak{u}_0).
\]
If $X=\mathbb{R}$, then $\mathfrak{F}_0 (\mathbb{R})$ coincides with the classical Banach space $c_0$. The groups of the form $\mathfrak{F}_0 (X)$  were  thoroughly studied in \cite{DMPT, Ga7, Ga8}.
In particular, $\mathfrak{F}_0$ is finitely multiplicative \cite{Ga7}.

Clearly, for every natural number $n$, the group $X^n$ is both a direct summand and a quotient group of $\mathfrak{F}_0 (X)$. So, if $\mathfrak{F}_0(X)$ has a topological property of one of the types  $\mathcal{P}_s$ or $\mathcal{P}_q$, then $X^n$ has the same property for every $n\in\NN$. Thus, if $\mathfrak{F}_0$ preserves a property $\mathcal{P} \in \mathcal{P}_s \cup \mathcal{P}_q$, then $\mathcal{P}$ must be preserved by any finite cartesian power of $X$.

The functor $\mathfrak{F}_0$ preserves many important topological properties.
\begin{fact} \label{fTopPropF}
Let $X\in \mathbf{TAG}$. Then
\begin{enumerate}
\item[{\rm (i)}] {\rm (\cite[3.4]{DMPT})} Let $\mathcal{P}$ denote one of the following topological properties from $\mathcal{P}_s \cup \mathcal{P}_q$: completeness, sequential completeness, metrizability, separability, maximal almost periodicity, local quasi-convexity and connectedness.  Then $\mathfrak{F}_0 (X)$ has $\mathcal{P}$ if and only if $X$ has $\mathcal{P}$.
\item[{\rm (ii)}] {\rm (\cite[3.1]{DMPT})} $\mathfrak{F}_0 (X)$ is precompact if and only if $X=\{ 0\}$.
\item[{\rm (iii)}] {\rm (\cite{Ga8})} $\mathfrak{F}_0 (X)$ is locally precompact if and only if $X$ is discrete.
\end{enumerate}
\end{fact}

Let $\mathbf{u} =\{ u_n \}_{n\in\omega}$ be a sequence in the dual group $\widehat{X}$ of  an Abelian topological group $X$ and let $H$ be a subgroup of $X$. Following \cite{DMT}, we set
\[
s_{\mathbf{u}} (X):=\{ x\in X: \; (u_n, x)\to 1 \},
\]
and say that a subgroup $H$ is $\mathfrak{g}$-{\it closed} in  $X$ if
\[
H= \bigcap_{ \mathbf{u} \in \widehat{X}^\NN } \left\{ s_{\mathbf{u}} (X) : \; H \leq s_{\mathbf{u}} (X) \right\}.
\]
Some properties of the closure operator $\mathfrak{g}$ are given in \cite{Dik-cl}.

The question of whether $c_0(X)$ is a $\mathfrak{g}$-closed subgroup of $X^\NN$ was considered in \cite{Ga7}. For the groups with the Schur property we have:
\begin{fact} {\rm (\cite{Ga7})} \label{fGclosed}
Let $X$ be a MAP Abelian group. If $X$ has the Schur property, then $c_0(X)$ is a $\mathfrak{g}$-closed subgroup of $X^\NN$.
\end{fact}
Below (see Theorem \ref{t12}) we show that in general the $\mathfrak{g}$-closeness of $c_0(X)$ in $X^\NN$ does not imply that $X$ has the Schur property.

Our interest in the study of the functor $\mathfrak{F}_0 $ is motivated  also by the following arguments. Firstly, if $X$ is metrizable and connected compact Abelian group, then the Bohr modification of $\mathfrak{F}_0 (X)$ is a connected precompact metrizable non-Mackey group  \cite{DMPT} (the definition of Mackey groups and their basic properties see \cite{CMPT}). Secondly, Rolewicz \cite{Rol} observed that the complete metrizable group $\mathfrak{F}_0 (\mathbb{T})$ is monothetic (see a proof of this fact in \cite[pp. 20-21]{DPS}). So monothetic Polish groups need not be compact nor discrete. In particular, these two arguments show that the groups of the form $\mathfrak{F}_0 (X)$ represent a nice source of (counter)examples in different areas of topological algebra. Further, the group $\mathfrak{F}_0 (\mathbb{T})$ plays a significant role in the theory of characterized subgroups of compact Abelian groups (see \cite{Ga, Ga3}). The next result especially emphasizes the importance of the study of the functor $\mathfrak{F}_0$:
\begin{fact} \label{f1}  {\rm (\cite{Ga8})}
For every LCA group $X$, the group $\mathfrak{F}_0 (X)$ is reflexive.
\end{fact}
So the functor $\mathfrak{F}_0$ preserves reflexivity on the class $\mathbf{LCA}$. On the other hand, excepting the case of discrete $X$, $\mathfrak{F}_0 (X)$ is never locally precompact by Fact \ref{fTopPropF}(iii). Hence we obtain a big class of reflexive complete Abelian groups which is beyond $\mathbf{LCA}$.



{\bf The main results}.
We emphasize that the study of the functor $\mathfrak{F}_0$, as well as other functors defined in a reasonable way in $\mathbf{TAG}$, in the light of Problems \ref{problem1}-\ref{problem4} is important  not only by the natural categorical reasons, but
also because we obtain in this way a {\it class} of groups with or without given topological properties.

The main goal of the article is  the study of both  important topological properties  preserved under $\mathfrak{F}_0 $ (Theorems \ref{tTopF} and \ref{t12}), and the preservation  under $\mathfrak{F}_0 $ of respected properties in the class $\mathbf{LCA}$ (Theorem \ref{t11}).

In the first main result  we complete the list of general topological properties which are preserved under the functor $\mathfrak{F}_0$ stated in Fact \ref{fTopPropF}, see Problems \ref{problem1} and  \ref{problem2} (all relevant definitions are given in Section \ref{secTop}).
\begin{theorem} \label{tTopF}
Let $X$ be an Abelian topological group. Then:
\begin{enumerate}
\item[{\rm (i)}] $X$ is an $(E)$-space if and only if $\mathfrak{F}_0(X)$ is an $(E)$-space.
\item[{\rm (ii)}] $X$ is a strictly  angelic space if and only if $\mathfrak{F}_0(X)$ is  strictly  angelic.
\item[{\rm (iii)}] $X$ is a \v{S}-space if and only if $\mathfrak{F}_0(X)$ is a \v{S}-space.
\end{enumerate}
\end{theorem}
On the other hand,  under additional hypothesis and taking into account that $\mathfrak{F}_0$ contains $X\times X$ as a closed subgroup, the functor $\mathfrak{F}_0$ does not preserve countable compactness,  sequentiality and  Fr\'{e}chet-Urysohness (see Remark \ref{r-TP} below). Further, since every pseudocompact group is precompact \cite{ArT}, $\mathfrak{F}_0(X)$ is pseudocompact if and only if  $X=\{ 0\}$ by Fact \ref{fTopPropF}(ii).

In the next theorem we give a complete answer to Problem \ref{problem4} for $F=\mathfrak{F}_0$ and various properties in $\pmb{\mathbb{A}}=\mathbf{LCA}$.
\begin{theorem} \label{t11}
Let $X$ be a LCA group. Then the following are equivalent:
\begin{enumerate}
\item[{\rm (i)}] $X$ is totally disconnected.
\item[{\rm (ii)}]  $\mathfrak{F}_0 (X)$  embeds into the product of a family of LCA groups.
\item[{\rm (iii)}]  $\mathfrak{F}_0 (X)$  is a nuclear group.
\item[{\rm (iv)}]  $\mathfrak{F}_0 (X)$  is a  Schwartz groups.
\item[{\rm (v)}] $\mathfrak{F}_0 (X)$ respects compactness.
\item[{\rm (vi)}] $\mathfrak{F}_0 (X)$ has the Schur property.
\item[{\rm (vii)}] $\mathfrak{F}_0 (X)$ respects countable compactness.
\item[{\rm (viii)}] $\mathfrak{F}_0 (X)$ respects  sequential compactness.
\item[{\rm (ix)}] $\mathfrak{F}_0 (X)$ respects  pseudocompactness.
\item[{\rm (x)}] $\mathfrak{F}_0 (X)$ respects functional boundedness.
\end{enumerate}
Moreover,  if $E$ is a functionally bounded subset in $\mathfrak{F}_0(X)^+$, then its closure $\mathrm{cl}_{\mathfrak{F}_0(X)} (E)$ is  compact in $\mathfrak{F}_0(X)$.
\end{theorem}
Concerning the equivalence of items (ii) and (v) we note that, Remus and Trigos-Arrieta \cite{ReT2} found Abelian topological groups which cannot be embedded into products of LCA groups though they are reflexive and respect compactness.

Now we formulate the third main theorem of the article:
\begin{theorem} \label{t12}
Let $X$ be a  compact connected metrizable Abelian group. Then:
\begin{enumerate}
\item[{\rm (i)}]  $\mathfrak{F}_0 (X)$ is a  monothetic connected Polish non-Schwartz Abelian group.
\item[{\rm (ii)}] $\mathfrak{F}_0 (X)$ has no the Schur property, and hence it does not respect compactness.
\item[{\rm (iii)}] $c_0\left( \mathfrak{F}_0 (X)\right)$ is $\mathfrak{g}$-closed in $\mathfrak{F}_0 (X)^\NN$.
\end{enumerate}
\end{theorem}
In the next remark we comment Theorem \ref{t12}:
\begin{remark} {\em
{\rm (1)}  In item (i) we generalize the above-mentioned result of Rolewicz \cite{Rol}. Note that our proof differs from Rolewicz's one even for $X=\TT$.

{\rm (2)} Remus and Trigos-Arrieta in \cite{ReT} provided examples of reflexive Abelian groups which do not respect compactness. They gave counterexamples to a result published by Venkataraman in \cite{Ven}. However a stronger result was known. Namely, the space $c_0$ is reflexive by \cite{Smith}  and does not have even the Schur property (see \cite[Exercise 3 on p. 212]{Dies}).  Items (i) and (ii) and Fact \ref{f1} show that  there exists reflexive monothetic Polish groups which do not respect sequentiality.

{\rm (3)} In connection with Fact \ref{fGclosed} we ask in Problem 49 of \cite{Ga7}: does $\mathbf{g}$-closeness of $c_0(X)$ in $X^\NN$ imply that $X$ has the Schur property?  The items (ii) and (iii) and Fact \ref{f1} show that in general the answer to this problem is negative even for reflexive Polish groups.}
\end{remark}

The article is organized as follows. In Section \ref{secN} we formulate and prove some important auxiliaries results which are used essentially in the article.
We prove Theorem \ref{tTopF} in Section \ref{secTop}. In Section \ref{secRes} we discuss and prove some general results concerning respected properties in topological groups. In particular, we introduce a new property (the $SB$-property) which gives a sufficient condition for topological groups to have the Glicksberg property etc. (see Theorem \ref{t-Res}). Theorem \ref{t11} is proved in Section \ref{secF},  and Theorem \ref{t12} we prove in Section \ref{sec2}. In Section \ref{secCateg} we show that the  Glicksberg property and the Schur property can be naturally defined by some functors in $\mathbf{TAG}$. In the last section we define another functors naturally coming from Functional Analysis and pose a dozen open problems.

\section{Auxiliary results} \label{secN}

\subsection{Uniformly discrete and bounded subsets in topological groups}

 The closure of a subset $E$ of a topological space $Y$ we denote by $\overline{E}$ or $\mathrm{cl}_Y(E)$.

Let $X$ be  a topological group  and $U\in \mathcal{N}(X)$. Recall that a subset $A$ of $X$  is called {\it left (right) $U$-separated} if $aU \cap bU = \emptyset$ (respectively, $Ua \cap Ub = \emptyset$) for every distinct elements $a,b\in A$. A subset $E$ of $X$ is called {\it left (right) uniformly discrete} if there is $U\in \mathcal{N}(X)$ such that $E$ is left (right) $U$-separated.

We omit proofs of the next two simple lemmas.
\begin{lemma} \label{l-UD}
Every left (right) uniformly discrete subset $A$ of a topological group $X$ is closed.
\end{lemma}


\begin{lemma} \label{lMax}
Let $E$ be a subset of a topological group $X$. Then, for every $U\in\mathcal{N}(X)$, there exists a maximal (under inclusion) left (right) $U$-separated subset of $E$.
\end{lemma}


Recall that, a subset $E$ of a topological group $X$ is called {\it left-precompact} (respectively, {\it right-precompact, precompact}) if, for every $U\in \mathcal{N}(X)$, there exists a finite subset $F$ of $X$ such that $E\subseteq F\cdot U$ (respectively, $E\subseteq U\cdot F$, $E\subseteq F\cdot U$ and $E\subseteq U\cdot F$). If $E$ is symmetric the three different definitions coincide.

In the sequel we essentially use the next useful proposition which holds true for every uniform space (see \cite{BGP}). For the sake of completeness we prove it.

\begin{proposition} \label{p-UD}
Let $E$ be an infinite subset of  a topological group $X$. Then
$E$ is either left (right) precompact or it has a countably infinite left (right) uniformly discrete subset.
\end{proposition}

\begin{proof}
Assume that $E$ contains a countably infinite left uniformly discrete subset $A$. Let us show that $E$ is not left precompact.
Suppose for a contradiction that $A$ is left precompact. Take a symmetric $U\in \mathcal{N}(X)$ such that $A$ is left $(U\cdot U)$-separated. As  $E$ is left precompact, there is a finite subset $F$ of $X$ such that $E\subset F\cdot U$. In particular, for every $a\in A$, there are $f_a \in F$ and $u_a\in U$ such that $a =f_a u_a$.

We claim that $f_a \not= f_b$ for every distinct $a, b \in A$. Indeed, otherwise, $a^{-1}b = (f_a u_a)^{-1} (f_a u_b) \in U\cdot U$. Hence $A$ is not  left $(U\cdot U)$-separated, a contradiction.
This means that the map $a\mapsto f_a$ is injective. As $F$ is finite, we obtain that also $A$ is finite. This contradiction shows that $E$ is not left precompact.

Let now $E$ do not have countably infinite left uniformly discrete subsets. We claim that $E$ is left precompact. Indeed, let $U\in \mathcal{N}(X)$. Take a symmetric $V\in \mathcal{N}(X)$ such that $V\cdot V \subseteq U$. By Lemma \ref{lMax}, choose a maximal left $V$-separated subset $F$ in $E$. By assumption $F$ is finite. Let us show that $E\subseteq F\cdot U$. Suppose for a contradiction that there is $g\in E\setminus  (F\cdot U)$. Then, for every $f\in F$, we have $g\cdot V \cap f\cdot V=\emptyset$ since otherwise,   $g \in f(V\cdot V) \subset F\cdot U$. This means that the set $F':= \{ g\} \cup F$ is left $V$-separated that  contradicts the maximality of $F$ to be a left $V$-separated subset of $E$. So $E\subseteq F\cdot U$. Thus $E$ is left precompact.
\end{proof}


Recall (see \cite{Tka87}) that a subset $E$ of a topological space $Y$ is called {\it functionally bounded in} $Y$ if every continuous real-valued function on $Y$ is bounded on $E$. The family of  all functionally bounded subsets of $Y$ we denote by $\mathcal{FB}(Y)$. We will use the following trivial fact:
\begin{lemma} \label{l-Boun}
Let $Y$ and $Z$ be topological spaces and let $f:Y\to Z$ be a continuous mapping. Then:
\begin{enumerate}
\item[{\rm (1)}] The family $\mathcal{FB}(Y)$ contains all compact subsets of $Y$ and it is closed under taking closure, finite unions and arbitrary subsets.
\item[{\rm (2)}] $f(\mathcal{FB}(Y))\subseteq \mathcal{FB}(Z)$.
\end{enumerate}
\end{lemma}

We prove the next lemma for the sake of completeness.
\begin{lemma} \label{l-FB}
Let $A$ be a left (right) uniformly discrete subset of a topological group $X$. Then $A$ is functionally bounded if and only if it is finite.
\end{lemma}
\begin{proof}
If $A$ is finite it is trivially functionally bounded.

Assume now that $A$ is functionally bounded in $X$.
Let $A$ be left $(U\cdot U)$-separated for a symmetric $U\in \mathcal{N}(X)$. Take a continuous function $f: X\to [0,1]$ such that $f(e)=1$ and $f|_{X\setminus U} =0$ \cite[8.4]{HR1}. Suppose for a contradiction that $A$ is infinite. Choose arbitrarily a one-to-one sequence $\{ a_n\}_{n\in\NN}$ in $A$ and set
\[
g(x):= \sum_{n\in\NN} n\cdot f(a^{-1}_n \cdot x), \quad \forall x\in X.
\]
Note that $g(x)$ is well-defined since $f(a^{-1}_n \cdot x)\not= 0$ only if $x\in a_n U$. Further, for every $x\in X$, the set $xU$ intersects with $a_n U$ for at most one $n\in\NN$. So $g(x)$ is continuous on $X$ and unbounded on $A$. Hence $A$ is not functionally bounded in $X$, a contradiction. Thus $A$ is finite.
\end{proof}

As a corollary of Proposition \ref{p-UD} and Lemma \ref{l-FB} we obtain:
\begin{proposition} {\rm (\cite[Statement A]{SiT})} \label{p-Bo}
Every functionally bounded subset $E$ of a topological group $X$ is precompact.
\end{proposition}

\begin{proof}
By Lemmas \ref{l-Boun}(1) and \ref{l-FB}, the set $E$ does not have infinite left (right) uniformly discrete subsets. Now Proposition \ref{p-UD} implies that $E$ is left  and right precompact. Thus $E$ is precompact.
\end{proof}

In what follows we use the next fact:
\begin{fact} \label{f-Precompact} {\rm (\cite[3.7.10]{ArT})}
Let $X$ be a (Ra\u{\i}kov) complete topological group. Then a subset $A$ of $X$ is precompact  if and only if its closure $\overline{A}$ is compact.
\end{fact}

Proposition \ref{p-Bo} and Fact \ref{f-Precompact} immediately implies:
\begin{corollary} \label{c-Bo}
Let $E$ be a subset of a complete topological group $X$. Then $E$ is functionally bounded iff $E$ is precompact iff the closure $\overline{E}$ of $E$ is compact.
\end{corollary}


\subsection{The Bohr topology on MAP  groups}

For a subset $A$ of  an  Abelian topological group $X$, {\it the annihilator} of $A$ in $\widehat{X}$ is $A^\perp :=\{ u\in\widehat{X}: (u,x)=1, \forall x\in A\}$.
The group $\widehat{X}$ endowed with the compact-open topology we denote by $X^\wedge$.

Recall that a subgroup $H$ of an Abelian topological group $X$ is called {\it dually closed} in $X$ if for every $x\in X\setminus H$ there exists a character $\chi \in H^\perp$ such that $(\chi,x)\not= 1$, i.e., $H=\cap_{\chi\in H^\perp} \ker(\chi)$. Clearly, every dually closed subgroup is closed. A subgroup $H$ is called {\it dually embedded} in $X$ if every continuous character of $H$ can be extended to a  continuous character of $X$.
\begin{fact} \label{fOpen} {\rm  (\cite[Lemma 3.3]{Nob2})}
Every open subgroup of an Abelian topological group $X$ is dually closed and dually embedded in $X$.
\end{fact}
Note that reflexive groups may contain closed non-dually closed subgroups, see \cite[Example (ii)]{Nob2}.

An Abelian group $X$ endowed with the discrete topology we denote by $X_d$. If $H$ is a dense subgroup of $(X_d)^\wedge$, we denote by $\sigma(X,H)$ or $T_H$ the weakest group topology on $X$ for which the elements of $H$ are continuous.
In particular, for each $X\in \mathbf{MAPA}$  the sets of the form
\[
\{ x\in X: \; (\chi, x) \in \TT_+ \},
\]
where $\chi\in \widehat{X}$ and $\TT_+ :=\{z\in  \TT:\ {\rm Re}(z)\ge 0\}$, form a subbase of the weak topology $\sigma(X, \widehat{X})$ on $X$.

\begin{fact} {\rm  (\cite{CoR})} \label{fBohr}
Let $X$ be an Abelian group and let $H_1$ and $H_2$ be dense subgroups of $(X_d)^\wedge$. Then $T_{H_1} \subseteq T_{H_2}$ if and only if $H_1 \subseteq H_2$. Hence $T_{H_1} = T_{H_2}$ iff
 $H_1 = H_2$.
\end{fact}

The next lemma generalizes some well-known results for LCA groups, see for example \cite{TriAr91, Tri91}. Item (1) and the necessity in item (4) see in \cite{Dik-cl} and \cite[2.1]{Aus2}. Below we generalize this lemma to non-Abelian MAP groups. For the convenience of the reader we give its complete proof.
\begin{lemma} \label{lBohr}
Let $H$ be a subgroup of a MAP Abelian group $X$. Then
\begin{enumerate}
\item[{\rm (1)}]  {\rm (\cite{{Dik-cl}})}  $H$ is dually closed in $X$ if and only if $H$ is closed in $X^+$.
\item[{\rm (2)}] If $H$ is a dually closed subgroup of $X$, then $(X/H)^+ =X^+/H$.
\item[{\rm (3)}] $\sigma(X,\widehat{X})|_H \leq  \sigma(H,\widehat{H})$.
\item[{\rm (4)}] {\rm (\cite{Her})} $H$ is dually embedded in $X$ if and only if $H^+ = \left(H, \sigma(X,\widehat{X})|_H\right)$.
\item[{\rm (5)}]  $H$ is dually closed and dually embedded in $X$ if and only if $H^+ $ is a closed subgroup of $X^+$.
\end{enumerate}
\end{lemma}

\begin{proof}
(1) Let $H$ be dually closed in $X$. Then $H=\cap_{\chi\in H^\perp} \ker(\chi)$. So $H$ is closed in $\sigma(X,\widehat{X})$.

Conversely, let $H$ be closed in $\sigma(X,\widehat{X})$ and $x\in X\setminus H$. Denote by $q^+ : X^+ \to X^+ /H$ the quotient map. Then $X^+ /H$ is a precompact group and  $q^+ (x)\not= 0$. It is well-known that there exists $\eta\in (X^+ /H)^\wedge$ such that $(\eta, q^+ (x))\not= 1$. So $\eta\circ q^+  \in \widehat{X^+} =\widehat{X}$ and $(\eta\circ q^+ , x)=(\eta, q^+ (x))\not= 1$.  Thus $H$ is dually closed in $X$.

(2) By item (1), $H$ is a closed subgroup of $X^+$. Hence $X/H$ and $X^+ /H$ are well-defined. Let $i: X/H \to X^+ /H$ be the identity continuous isomorphism and let $q: X\to X/H$ be the quotient map. By Fact \ref{fBohr}, we have to show that $\widehat{X/H}=\widehat{X^+/H}$. Since $i$ is continuous we have $\widehat{X^+/H} \subseteq \widehat{X/H}$.

Let us prove the converse inclusion.
Fix $\eta\in \widehat{X/H}$. Then $\eta\circ q\in \widehat{X}$. So $\eta\circ q\in \widehat{X^+}$. Since $\eta\circ q|_H =0$ and $H$ is closed in $X^+$, we obtain $\eta\in \widehat{X^+/H}$. Thus $\widehat{X/H} \subseteq \widehat{X^+/H}$. Therefore, $\widehat{X/H}=\widehat{X^+/H}$ and $(X/H)^+ =X^+/H$.

(3) Since $\chi|_H \in \widehat{H}$ for every $\chi\in \widehat{X}$, the assertion follows from  Fact \ref{fBohr}.

(4) Assume that $H$ is dually embedded in $X$ and $\eta\in \widehat{H}$. Denote by $\chi$ an arbitrary extension of $\eta$ to a continuous character of $X$. Then
\[
\{ h\in H : \ (\eta, h) \in \TT_+ \} = H \cap \{ x\in X : \ (\chi, x) \in \TT_+ \}.
\]
So $\sigma(H,\widehat{H}) \subseteq \sigma(X,\widehat{X})|_H$. The converse inclusion follows from item (3). Thus $\sigma(H,\widehat{H}) = \sigma(X,\widehat{X})|_H$. Therefore, $H^+ = \left(H, \sigma(X,\widehat{X})|_H\right)$.

Let now $H^+ = \left(H, \sigma(X,\widehat{X})|_H\right)$. So  $\sigma(H,\widehat{H}) = \sigma(X,\widehat{X})|_H =\sigma(H,\widehat{X}|_H)$. It follows from  Fact \ref{fBohr} that $\widehat{H}= \widehat{X}|_H$. This explicitly means that every $\eta\in \widehat{H}$  can be extended to $\chi\in \widehat{X}$. Thus $H$ is dually embedded in $X$.

(5) follows from items (1) and (4).
\end{proof}

The next lemma extends Lemma 4.4 in \cite{Tri91}.
\begin{lemma} \label{l-FbMAP}
Let $H$ be an open subgroup of a MAP Abelian group $X$. If $E$ is a functionally bounded subset in $X^+$,  then $E$ is contained in a finite union of cosets of $H$.
\end{lemma}

\begin{proof}
By Fact \ref{fOpen},  $H$ is dually closed in $X$. Hence $H$ is closed in $X^+$ and $(X/H)^+ = X^+/H$ by Lemma \ref{lBohr}. Let  $\pi: X^+ \to X^+/H$ be the quotient map. Then $\pi(E)$ is functionally bounded subset in $X^+/H$ (see Lemma \ref{l-Boun}). We have to show that $\pi(E)$ is finite.
As $X/H$ is discrete and $\pi(E)$ is functionally bounded in  $(X/H)^+$, $\pi(E)$ is finite by \cite[Lemma 4.4]{Tri91}.
\end{proof}


Now we consider the Bohr compactification of non-Abelian groups.
Let $X$ be a Hausdorff topological group. Recall that a compact group $bX$ is called the {\it Bohr compactification} of $X$ if there exists a continuous homomorphism $i$ from $X$ onto a dense subgroup of $bX$ such that the pair $(bX,i)$ satisfies the following {\it universal property}: If $p:X\to C$ is a continuous homomorphism into a compact group $C$, then there exists a continuous homomorphism $j^p: bX \to C$ such that $p=j^p\circ i$. Following von Neumann \cite{Neu},  the group $X$ is called {\it maximally almost periodic} (MAP) if the group $X^+$ is Hausdorff, where $X^+$ is the group $X$ endowed with the topology induced from $bX$. The family $\mathbf{MAP}$ of all MAP topological groups is a subcategory of the category $\mathbf{TG}$ of all Hausdorff topological groups and continuous homomorphisms.

Note that every  irreducible representation of a compact group is finite-dimensional (see  \cite[22.13]{HR1}). It is well-known also that we can identify  the set of all finite-dimensional irreducible representations of a MAP {\it Abelian} group $X$ with the usual set  of its continuous characters. 

Let $X$ be a MAP group. Denote by $\widehat{X}$  the set of all finite-dimensional irreducible representations of $X$. For a  finite-dimensional  irreducible representation $u \in \widehat{X}$ of $X$ by unitary operators on a Hilbert space $\mathcal{H}$ and an arbitrary $x\in X$ we denote by $(u,x)$ the operator $u(x)$ on the representation space $\mathcal{H}$ of $u$. The identity operator on $\mathcal{H}$ is denoted by $I_\mathcal{H}$ (or just by $I$).
The next folklore lemma easily follows from the definition of Bohr compactification and irreducible representation and Corollary 3.6.17 of \cite{ArT} (we denote by $\overline{X}$ the completion of $X$):
\begin{lemma} \label{l-BohrRep}
If $X$ is a MAP group, then $b\overline{X}=bX$ and
$
\widehat{X}=\widehat{\overline{X}}=\widehat{X^+} =\widehat{bX}.
$
\end{lemma}

Next fact immediately follows from the definition of the Bohr compactification (cf. \cite{BaMP2}):
\begin{fact} \label{f-Bohr} 
Let $\phi: X\to Y$ be a continuous homomorphism between MAP groups. Then $\phi: X^+ \to Y^+$ is also continuous.
\end{fact}

Note that the definitions of the dual closure and the dual embedding are also transferred to non-Abelian case without  modifications.
A subgroup $H$ of a $MAP$ group $X$ is called {\it dually closed } in $X$ if for every $g\in X\setminus H$ there exists an irreducible (finite-dimensional) representation $\chi\in H^\perp := \{ \eta \in \widehat{X} : \eta|_H =I\}$ such that $(\chi,g)\not= I$ (\cite{Ga2}). And $H$ is named {\it dually embedded} in $X$ if every $\chi\in \widehat{H}$ can be extended to an irreducible (finite-dimensional) representation of $X$ (\cite{Her}).
The next lemma generalizes items (1), (4) and (5) of  Lemma \ref{lBohr}:
\begin{lemma}
Let $H$ be a subgroup of a MAP  group $X$ and $p:H\to X$ be the natural embedding. Then
\begin{enumerate}
\item[{\rm (1)}] $H$ is dually closed in $X$ if and only if $H$ is a normal subgroup of $X$ and $H$ is closed in $X^+$.
\item[{\rm (2)}] {\rm (\cite{Her})} $H$ is dually embedded in $X$ if and only if $j^p|_H : H^+ \to X^+$ is an embedding.
\item[{\rm (3)}]  $H$ is dually closed and dually embedded in $X$ if and only if $H^+ $ is a closed subgroup of $X^+$.
\end{enumerate}
\end{lemma}

\begin{proof}
(1) Let $H$ be dually closed in $X$. Then $H=\cap_{\chi\in H^\perp} \ker(\chi)$. So $H$ is a normal subgroup of $X$ and it is closed in $X^+$.

Conversely, let $H$ be closed and normal in $X^+$ and let $x\in X\setminus H$. Denote by $q^+ : X^+ \to X^+ /H$ the quotient map. Then $X^+ /H$ is a precompact group and  $q^+ (x)\not= e$. Since $X^+/H$ has a compact completion \cite[3.7.16]{ArT}, Peter-Weyl's theorem implies that there exists $\eta\in \widehat{X^+ /H}$ such that $(\eta, q^+ (x))\not= I$. So, by Lemma \ref{l-BohrRep}, $\eta\circ q^+  \in \widehat{X^+} =\widehat{X}$ and $(\eta\circ q^+ , x)=(\eta, q^+ (x))\not= I$.  Thus $H$ is dually closed in $X$.

(3) follows from items (1) and (2).
\end{proof}


In what follows we use the next property:
\begin{definition}
A  topological group $X$ is said to have an {\em $\SP$-property} if every separable precompact subset of $X$ has compact closure.
\end{definition}
If a topological group $X$ is complete, the closure $\overline{E}$ of each precompact subset $E$ is compact by Fact \ref{f-Precompact}. So we obtain:
\begin{proposition}
Every complete topological group has the $\SP$-property.
\end{proposition}

Recall that a (Hausdorff) topological space $Y$ is called {\it sequentially compact} if every sequence in $Y$ contains a convergent subsequence.
\begin{example} {\em
There is a sequentially compact non-compact Abelian group $H$ which has the $\SP$-property. We use the notion of $\Sigma$-product introduced by Pontryagin. Let $G:= X^\kappa$, where $X$ is a metrizable compact Abelian group and the cardinal $\kappa$ is uncountable. For $g=(x_i)_{i\in\kappa} \in G$, denote $\mathrm{supp}(g):=\{ i\in\kappa: x_i \not= 0\}$ and set
\[
H:= \{ g\in G: \; |\mathrm{supp}(g)| \leq\aleph_0 \}.
\]
Then $H$ with the induced topology is a dense subgroup of $G$. We claim that $H$  is a sequentially compact group with the $\SP$-property. Indeed,  any countable subset of $H$ is contained in a countable product $Y$ of copies of $X$. Since $Y$  is a compact and metrizable subgroup of $X$, we obtain that the group $H$ is sequentially compact  with the $\SP$-property. Note also that $H$ is Fr\'{e}chet-Urysohn by \cite{Nob}. }
\end{example}

\subsection{Some results on groups of the form $\mathfrak{F}_0(X)$}

Let $X$ be an Abelian topological group. For every $n\in\NN$, define an injective homomorphism $\nu_n: X\to X^{(\NN)}$ by
\[
 \nu_n (x)=(0,\dots,0,x,0\dots),
\]
where $x\in X$ is placed in position $n$. Define the projections $p_n: X^\NN \to X^n$ and $\pi_n :X^\mathbb{N} \to X$ by
\[
p_n (\mathbf{x})  :=(x_1,\dots, x_n), \ \mbox{ and } \  \pi_n (\mathbf{x}) := x_n , \ \forall \mathbf{x} =(x_n)_{n\in\NN} \in X^\NN.
\]
Clearly,  $p_n$ and $\pi_n$ are continuous in the uniform topology $\mathfrak{u}$. We use the same notations for the restrictions of $p_n$ and $\pi_n$ onto $c_0(X)$.

Let $\{ E_n \}_{n\in\NN}$ be a sequence of  subsets of an Abelian  topological group $X$. Following \cite{Ga7}, we say that $\{ E_n \}$ is a {\it null-sequence} in $X$ if, for every $U\in\mathcal{N}(X)$, $E_n \subseteq U$ for all sufficiently large $n\in\NN$.

\begin{lemma} \label{lNull1}
If $\{ E_n \}_{n\in\NN}$ is a null-sequence of subsets in an Abelian   topological group  $X$, then $\{ \overline{E_n} \}_{n\in\NN}$ is a null-sequence in $X$ as well.
\end{lemma}

\begin{proof}
Let $U\in \mathcal{N}(X)$. Take $V\in \mathcal{N}(X)$ such that $\overline{V}\subseteq U$ \cite[4.7]{HR1}. Choose $m\in\NN$ such that $E_n \subseteq V$ for every $n\geq m$. Then $\overline{E_n} \subseteq \overline{V}\subseteq U$ for every $n\geq m$. Thus $\{ \overline{E_n} \}_{n\in\NN}$ is a null-sequence in $X$.
\end{proof}

\begin{lemma} \label{lNull2}
Let $E$ be a subset of $\mathfrak{F}_0(X)$ for an Abelian topological group $X$. If  $\{ \pi_n (E) \}_{n\in\NN}$ is not a null-sequence in $X$, then $E$ has an infinite uniformly discrete subset.
\end{lemma}

\begin{proof}
By assumption, there exists $U\in \mathcal{N}(X)$ such that $\pi_n (E) \not\subset U$ for an infinite set of indices. Choose a symmetric $V\in \mathcal{N}(X)$ such that $V+V \subseteq U$. We shall build a uniformly discrete sequence in $E$ by induction.

Take $n_1 \in\NN$ such that $\pi_{n_1} (E) \not\subset U$, and choose $\mathbf{b}_1 =(b_n^1)_{n\in\NN} \in E$ such that $\pi_{n_1} (\mathbf{b}_1)=b_{n_1}^1 \not\in U$. Since $\mathbf{b}_1 \in c_0(X)$, there is an index $j_1, j_1 > n_1$, such that $b_i^1 \in V$ for every $i\geq j_1$.

Take $n_2 \in\NN, n_2 >j_1,$ such that $\pi_{n_2} (E) \not\subset U$, and choose $\mathbf{b}_2 =(b_n^2)_{n\in\NN} \in E$ such that $\pi_{n_2} (\mathbf{b}_2)=b_{n_2}^2 \not\in U$. Since $\mathbf{b}_2 \in c_0(X)$, there is an index $j_2, j_2 > n_2$, such that $b_i^2 \in V$ for every $i\geq j_2$.

Continuing this process we shall build a sequence $\{ \mathbf{b}_n \}_{n\in\NN}$ in $E$. If $k<m$, then $\pi_{n_m} (\mathbf{b}_m - \mathbf{b}_k)= b_{n_m}^m -b_{n_m}^k \not\in V$ since, otherwise, $b_{n_m}^m \in b_{n_m}^k +V \subset V+V\subseteq U$ that contradicts the choice of $\mathbf{b}_m$. Thus the sequence $\{ \mathbf{b}_n \}$ is $V^\NN \cap c_0(X)$-separated.
\end{proof}

The next two results play an essential role in the sequel:

\begin{fact} \label{p11} {\rm (\cite{Ga7})}
Let $X$ and $Y$ be Abelian topological groups.
\begin{enumerate}
\item[{\rm (1)}] The groups $\mathfrak{F}_0(X) \times \mathfrak{F}_0(Y)$ and $\mathfrak{F}_0 (X\times Y)$ are topologically isomorphic.
\item[{\rm (2)}] If  $H$ is a (respectively, closed or open) subgroup of $X$, then $\mathfrak{F}_0 (H)$ is a (respectively, closed or open) subgroup of $\mathfrak{F}_0 (X)$.
\item[{\rm (3)}] A closed subset $K$ of $\mathfrak{F}_0(X)$ is compact if and only if the sequence $\{ \pi_n (K) \}_{n\in\NN}$ is a null-sequence of compact subsets of $X$. Moreover, if $K$ is  compact, then the product $\prod_{n\in\NN} \pi_n (K)$ is a compact subset of $\mathfrak{F}_0(X)$.
\end{enumerate}
\end{fact}

\begin{fact} {\rm (\cite{Ga8})} \label{f3}
Let $X$ be a LCA group and $X\cong \mathbb{R}^n \times X_0$, where $n\in\omega$ and $X_0$ has an open compact subgroup. Then
\begin{enumerate}
\item[{\rm (1)}]  $\mathfrak{F}_0 (X)\cong c_0^n \times \mathfrak{F}_0 (X_0)$ and $\mathfrak{F}_0 (X)^\wedge \cong \ell_1^n \times \mathfrak{F}_0 (X_0)^\wedge$.
\item[{\rm (2)}] $\widehat{\mathfrak{F}_0 (X_0)} \subseteq \widehat{X}^\NN_0$. Further,  $\mathbf{g} =(g_n)_{n\in\NN}\in \mathfrak{F}_0 (X_0)^\wedge$ if and only if  there exists an open subgroup $H$ of $X_0$ and a natural number $m$ such that $g_n \in H^\perp$ for every $n\geq m$. In this case we have
    \[
    (\mathbf{g} ,\xxx)= \lim_{n\to\infty} \prod_{i=1}^n (g_i, x_i) , \; \forall \xxx=(x_n)_{n\in\NN} \in \mathfrak{F}_0(X).
    \]
\end{enumerate}
\end{fact}

\section{Topological properties which are preserved under the functor $\mathfrak{F}_0$} \label{secTop}

Let $X$ be a topological space. Recall that a subset $A$ of  $X$ is called
\begin{enumerate}
\item[--] {\it relatively compact} if its closure ${\bar A}$ is compact;
\item[--] {\it relatively countably compact} if each countably infinite subset in $A$ has a cluster point in $X$;
\item[--] {\it relatively sequentially compact} if each sequence in $A$ has a subsequence converging to a point of $X$;
\item[--] {\it countably compact} or, respectively, {\it sequentially compact } if in the above two definitions the cluster point or, respectively, the limit point is required to be in $A$.
\end{enumerate}
Clearly, compact and sequentially compact subsets are countably compact. On the other hand, there are sequentially compact subsets of a completely regular Hausdorff space which are nonclosed, and their closure are not countably compact (see Examples 28 and 29 in \cite{BMPT}). So topological spaces, in which any two of these in general different topological properties  coincide,  are of independent interest. Let us recall some of them.

A  Hausdorff topological space $X$ is called
\begin{enumerate}
\item[(1)] an {\it  $(E)$-space} if its relatively countably compact subsets are relatively compact (see \cite[Exercise 1, p. 209]{Gro});
\item[(2)] a {\it \v{S}-space} if its compact subsets are sequentially compact;
\item[(3)] an {\it angelic space} if for every relatively countably compact subset $A$ of $X$ the following two claims hold:
    \begin{enumerate}
    \item[(i)] $A$ is relatively compact, and
    \item[(ii)] if $x\in {\bar A}$, then there is a sequence in $A$ which converges to $x$;
     \end{enumerate}
\item[(4)] a  {\it strictly angelic space} if it is angelic and each its separable compact subspace is first countable (\cite{Gov}).
\end{enumerate}
It is easy to see that all  classes of  spaces (1)-(4) are closed under taking closed subspaces, i.e., they are of $\mathcal{P}_s$-type. Note that the product of two countably compact spaces may not be countably compact \cite[3.10.19]{Eng}. On the other hand, the countable product of (sequentially) compact spaces is (sequentially) compact (see \cite[3.10.35]{Eng}). Also the class of strictly angelic spaces is countably productive \cite{Gov}. This explains why Theorem \ref{tTopF} is valid in the case when countable compactness coincides with (sequential) compactness. Clearly, every (strictly) angelic space is an $(E)$-space.

Now we are in position to prove Theorem \ref{tTopF} (we use notations from Section 2.3).

{\large Proof of Theorem  {\rm \ref{tTopF}}}.
Since all the classes of  groups (i)-(iii) in the theorem are closed under taking closed subgroups, we need to prove only the necessity in (i)-(iii).

(i) Assume that $X$ is an $(E)$-space. We have to show that $\mathfrak{F}_0(X)$ is also an $(E)$-space. To end this we have to prove that the closure ${\bar A}$ of  every relatively countably compact subset $A$ of $\mathfrak{F}_0(X)$ is compact.

For every $n\in\NN$, set $A_n := \pi_n (A)$. We claim that $A_n$ is relatively countably compact in $X$. Indeed, for every   countably infinite subset $\{ a^n_k\}_{k\in\omega} $ in $A_n$ take arbitrarily its preimage $\{ \mathbf{a}_k\}_{k\in\omega} $ in $A$, i.e., $\pi_n (\mathbf{a}_k)=a^n_k$ for every $k\in\omega$.  By definition, $\{ \mathbf{a}_k\}_{k\in\omega} $ has a cluster point $\mathbf{a}$ in $\mathfrak{F}_0(X)$. Clearly,  $\pi_n (\mathbf{a})$ is a cluster point of $A_n$.

Since $X$ is an $(E)$-space, we obtain that the closure $\overline{A_n}$ of $A_n$  is compact in $X$ for every $n\in\NN$.

Let us show that the sequence $\{ A_n\}_{n\in\omega}$ is a null sequence in $X$. Indeed, otherwise $A$ has a uniformly discrete sequence $\{ \mathbf{b}_k\}_{k\in\NN} $ by Lemma \ref{lNull2}. Since the set $\{ \mathbf{b}_k\}$ is closed  in $\mathfrak{F}_0(X)$ by Lemma \ref{l-UD}, it does not have cluster points.  Thus $A$ is not relatively countably compact, a  contradiction.

Lemma \ref{lNull1} implies that $\{ \overline{A_n}\}_{n\in\NN}$ is a null sequence of compact subsets of $X$. Set $K:= \prod_{n\in\NN} \overline{A_n}$. Fact \ref{p11}(3) yields that $K$ is a compact subset of $\mathfrak{F}_0(X)$. Since $A\subseteq K$ we obtain that ${\bar A}$ is compact in $\mathfrak{F}_0(X)$. Thus $\mathfrak{F}_0(X)$ is an $(E)$-space.

(ii)  Assume that $X$ is a strictly angelic space. We have to show that $\mathfrak{F}_0(X)$ is  strictly angelic as well.

Let us show first that every separable compact subset $K$ of $\mathfrak{F}_0(X)$ is first countable. For every $n\in\NN$, set $K_n := \pi_n (K)$ and put $K' := \prod_{n\in\NN} K_n$. By Fact \ref{p11}(3), $K'$ is a compact subset of  $\mathfrak{F}_0(X)$. Since $X$ is strictly  angelic, $K_n$ is first countable for every $n\in\NN$. Thus $K'$ and its compact subset $K$ are first countable (see \cite[2.3.14]{Eng}).

Let $B$ be a  relatively countably compact subset of $\mathfrak{F}_0(X)$. By item (i), $B$ is relatively compact in $\mathfrak{F}_0(X)$. So, to prove  that $\mathfrak{F}_0(X)$ is a strictly angelic space,  it is remained to show  that for  every $\mathbf{x}=(x_n)_{n\in\NN} \in {\bar B}$ there exists a sequence in $B$ converging to $\mathbf{x}$.

For every $n\in \NN$, set $C_n := \pi_n ({\bar B})$.  Then $\{ C_n\}_{n\in\NN}$ is a null sequence of compact subsets of $X$ by Fact \ref{p11}(3).

For every $n\in \NN$, set  $B_n :=p_n (B)$. As  $B$ is relatively compact in $\mathfrak{F}_0(X)$, the set $B_n$ is relatively compact in $X^n$ and $p_n(\xxx)\in \overline{B_n}$. Since $B_n$ is relatively countably compact and $X^n$ is strictly angelic \cite{Gov}, for every $n\in \NN$ there is a sequence $\{ x_{n,k} \}_{k\in\NN}$ in $B_n$ converging to $p_n(\xxx)$. Take arbitrarily a sequence $\{ \xxx_{n,k} \}_{k\in\NN}$ in $B$ such that $p_n(\xxx_{n,k})=x_{n,k}$. Set $S:= \{ \xxx_{n,k} \}_{n,k\in\NN}$ and $Z := \mathrm{cl}_{\mathfrak{F}_0(X)} (S)$. Then $S\subseteq B$ and  $Z$ is a separable compact subset of ${\bar B}$. As we proved above, $Z$ is first countable.  Thus, to show the existence of a sequence in $B$  converging to $\xxx$ it is enough to prove that $\xxx\in Z$.

Fix arbitrarily $U\in \mathcal{N}(X)$. Take a symmetric $V\in \mathcal{N}(X)$ such that $V+V\subseteq U$. Choose $m\in\NN$ such that $C_n \subset V$ for every $n> m$. Then $\pi_n(\xxx)$ and $\pi_n(\xxx_{l,k})$ belong to $V$ for every $n>m$ and each $l,k\in\NN$. So $\pi_n(\xxx_{l,k}) - \pi_n(\xxx) \in V+V\subseteq U$  for every $n>m$ and each $l,k\in\NN$. Take $k_0 \in\NN$ such that $p_m(\xxx_{m,k_0}) - p_m(\xxx) \in U^n$. Clearly, $\xxx_{m,k_0} - \xxx \in U^\NN \cap c_0(X)$. Thus $\xxx\in Z$.

(iii)  Assume that $X$ is a \v{S}-space and let $K$ be a compact subset of $\mathfrak{F}_0(X)$. We have to prove that $K$ is sequentially compact.

For every $n\in\NN$, set $K_n := \pi_n (K)$. Then $K_n$ is compact in $X$. By assumption, $K_n$ is sequentially compact for every $n\in\NN$. Hence $K' := \prod_{n\in\NN} K_n$  is  sequentially compact by \cite[3.10.35]{Eng}. Note that $K'$ is a compact subset of $\mathfrak{F}_0(X)$ by Fact \ref{p11}(3), and  $K\subseteq K'$. Thus $K$ is  sequentially compact by \cite[3.10.33]{Eng}. Therefore $\mathfrak{F}_0(X)$ is a  \v{S}-space.
$\Box$


Recall that a topological group $X$ is said to have a {\it subgroup topology} if it has a base at the identity consisting subgroups.

\begin{proposition} \label{pSubT}
Let $X$ be an Abelian topological group with  subgroup topology. Then
\begin{itemize}
\item[{\rm (i)}]  $\mathfrak{F}_0(X)$ has a subgroup  topology.
\item[{\rm (ii)}] $\mathfrak{F}_0(X)$ embeds into a product of discrete Abelian groups. In particular, $\mathfrak{F}_0(X)$ is nuclear.
\end{itemize}
\end{proposition}

\begin{proof}
(i) Let $\mathcal{B}$ be a base at zero consisting subgroups. Then $\{ U^\NN \cap c_0(X) : \ U\in \mathcal{B}\}$ is a base at zero in $\mathfrak{F}_0(X)$ consisting subgroups. Thus $\mathfrak{F}_0(X)$ has a subgroup topology.

(ii) By item (i), $\mathfrak{F}_0(X)$ has a subgroup topology. Hence $\mathfrak{F}_0(X)$ embeds  into a product of discrete Abelian groups (see, for example, Proposition 2.2 of \cite{AG}). By \cite[7.5, 7.6 and 7.10]{Ban},  $\mathfrak{F}_0(X)$ is nuclear.
\end{proof}


Let $X$ be an Abelian topological group. For a subset $A$ of $X$, we denote by $A^\triangleright$ the set $\{\chi\in X^\wedge:\ \chi(A)\subseteq \TT_+\}$.
 A subset $A$ of $X$ is called {\it quasi-convex} if for every $x\in X\setminus A$ there exists
  $\chi\in A^\triangleright$ which satisfies  $\chi(x)\notin \TT_+$.
An Abelian topological group is called {\it locally quasi-convex} if it admits a neighborhood base at the neutral element
$0$ consisting of quasi-convex sets.

The next proposition is an immediate corollary of Fact \ref{fTopPropF}(i) and \cite[Theorem 14]{BMPT}.
\begin{proposition}
If $X$ is a locally quasi-convex Abelian group which admits a coarser metrizable group topology, then  $\mathfrak{F}_0(X)^+$ is strictly angelic.
\end{proposition}


Denote by $\mathbf{TVS}$ (respectively, $\mathbf{LCS}$) the subcategory of $\mathbf{TAG}$ consisting of all real topological vector space (TVS for short) (respectively, real locally convex spaces, LCS for short). The following proposition shows that $\mathbf{TVS}$ and $\mathbf{LCS}$ are $\mathfrak{F}_0$-invariant. For a subset $A$  of  a TVS $L$ and an $\alpha\in\mathbb{R}$ we set $\alpha A:=\{ \alpha a \in L : a\in A\}$.
\begin{proposition} \label{p-TVS}
{\rm (i)} If $L$ is a real (respectively, complex) TVS, then  $\mathfrak{F}_0 (L)$ is a  real  (respectively, complex)  TVS as well.

{\rm (ii)} If $L$ is a real LCS, then  $\mathfrak{F}_0 (L)$ is also a real LCS.
\end{proposition}

\begin{proof}
(i) Set $\mathfrak{U}=\{ \mathbf{U}: U\in\mathcal{N}(L)\}$, where $\mathbf{U}:= U^\NN \cap c_0(L)$. We have to check the next three conditions (see \cite[\S 15.2]{Kothe}):
\begin{enumerate}
\item[(a)] For each $\mathbf{U}\in\mathfrak{U}$ there is a $\mathbf{V}\in\mathfrak{U}$ with $\mathbf{V}+\mathbf{V}\subseteq \mathbf{U}$.
\item[(b)] For each $\mathbf{U}\in\mathfrak{U}$ there is a $\mathbf{V}\in\mathfrak{U}$ for which $\alpha \mathbf{V}\subseteq \mathbf{U}$ for all $\alpha$ with $|\alpha|\leq 1$.
\item[(c)] For each $\mathbf{U}\in\mathfrak{U}$ and each $\mathbf{x}\in \mathfrak{F}_0(L)$ there is a $k\in\NN$ for which $\mathbf{x}\in  k\mathbf{U}$.
\end{enumerate}
Since $L$ is a TVS, take  $V\in\mathcal{N}(L)$ such that $V+V\subseteq U$ and $\alpha V\subseteq U$ for all $\alpha$ with $|\alpha|\leq 1$. Clearly, $\mathbf{V}$ satisfies (a) and (b).
Let us check (c).
For $\mathbf{x}=(x_n)_{n\in\NN} \in \mathfrak{F}_0(L)$ choose $m\in\NN$ such that $x_n\in V$ for every $n>m$. Since $L$ is a TVS, take $k\in\NN$ such that $x_1,\dots, x_m \in kV\subseteq kU$. Note that $\frac{1}{k}V\subseteq U$. So, for every $n>m$ we have $x_n =k\cdot (\frac{1}{k} x_n)\in kU$. Now it is clear that $\mathbf{x}\in  k\mathbf{U}$.

(ii) By item (i),  $\mathfrak{F}_0 (L)$ is a real TVS. Since $L$ is locally quasi-convex group by \cite[2.4]{Ban}, the group $\mathfrak{F}_0 (L)$ is also locally quasi-convex by  Fact \ref{fTopPropF}(i). Applying \cite[2.4]{Ban} once again we obtain that  $\mathfrak{F}_0 (L)$ is  a real locally convex space.
\end{proof}

Proposition \ref{p-TVS},  Fact \ref{fTopPropF}(i) and \cite[Corollary 16]{BMPT} immediately imply:
\begin{corollary}
If $L$ is a real complete LCS, then $\mathfrak{F}_0 (L)^+$ is an $(E)$-space.
\end{corollary}

\begin{remark} \label{r-TP} {\em
Under Martin's Axiom, van  Douwen  \cite{vDou} announced the  existence of countably  compact  topological  group $X$  for  which  $X^2$ is  not  countably  compact. Under the weaker assumption, Hart and van Mill  \cite{HaMill} constructed a topological group $X$ whose square is not countably compact. Their result was generalized by Tomita \cite{Tomit}. Shibakov  showed that, under CH, the square of a countable Fr\'{e}chet-Urycohn Abelian group can be not Fr\'{e}chet-Urycohn \cite{Shi} and even not sequential (see \cite{Sha}). Further results and questions in this direction see
\cite{Maly, Sha}.
So, under additional hypothesis and taking into account that $\mathfrak{F}_0$ contains $X^2$ as a closed subgroup, the functor $\mathfrak{F}_0$ does not preserve countable compactness,  sequentiality and  Fr\'{e}chet-Urysohness. We do not know whether there exists an angelic Abelian topological group whose square is not angelic. Note only that (under CH) \cite{BoRos} provides an example of a nonangelic product of two compact Hausdorff angelic spaces.}
\end{remark}


\section{Respected properties for topological groups} \label{secRes}

For a Tychonoff topological space $X$, we denote by $\mathcal{S}(X)$ ($\mathcal{C}(X)$, $\mathcal{SC}(X)$, $\mathcal{CC}(X)$, $\mathcal{PC}(X)$ and $\mathcal{FB}(X)$, respectively) the set of all converging sequences in $X$ with the limit point (respectively, the set of all compact, sequentially compact, countably compact, pseudocompact and functionally bounded in $X$ subsets of $X$). Note that continuous images and finite disjoint unions of compact sets (respectively, sequentially compact, countably compact, pseudocompact and functionally bounded in $X$ sets)  are compact sets (sequentially compact, countably compact, pseudocompact and functionally bounded in $X$ sets, respectively) (see \cite{Eng}).  The following diagram of inclusions  is well-known (see \cite[\S 3.10]{Eng}):
\begin{equation} \label{diag}
\begin{diagram}
\node[2]{\mathcal{C}(X)} \arrow{se,b}{}\\
\node{\mathcal{S}(X)} \arrow{ne,r}{}
\arrow{se,b}{}
\node[2]{\mathcal{CC}(X)} \\
\node[2]{\mathcal{SC}(X)} \arrow{ne,r}{}
\end{diagram}
\longrightarrow  \mathcal{PC}(X) \longrightarrow  \mathcal{FB}(X).
\end{equation}
In general, if $\mathcal{P}$ is a topological property and $X$ is a topological space, we denote by $\mathcal{P}(X)$ the set of all subspaces of $X$ with $\mathcal{P}$. In what follows we consider the following families of topological properties
\[
\mathfrak{P}_0 := \{ \mathcal{S}, \mathcal{C}, \mathcal{SC}, \mathcal{CC}, \mathcal{PC}\} \quad \mbox{ and } \quad \mathfrak{P} :=\mathfrak{P}_0 \cup \mathcal{FB}.
\]

Following \cite{ReT3}, we define:
\begin{definition} {\rm (\cite{ReT3})}
If $\mathcal{P}$ denotes a topological property, then we say that a $MAP$ group $X$ {\em respects} $\mathcal{P}$ if $\mathcal{P}(X)=\mathcal{P}(X^+)$.
\end{definition}
So, for example, $X$ respects sequentiality or compactness if $\mathcal{S}(X)=\mathcal{S}(X^+)$ or $\mathcal{C}(X)=\mathcal{C}(X^+)$, respectively.


Diagram \ref{diag} immediately implies the next simple necessary condition when a MAP group respects one of the properties from $\mathfrak{P}$:
\begin{proposition} \label{p-NoRes}
Let $X$ be a MAP topological group. If $X$ respects one of the properties from $\mathfrak{P}$, then every  convergent sequence in $X^+$  with the limit point is functionally bounded in $X$. In other words, if  $X^+$ has a convergent sequence  with the limit point which is not functionally bounded in $X$, then $X$ does not respect any property from $\mathfrak{P}$.
\end{proposition}

This proposition and Diagram \ref{diag} motivate the following definition.
\begin{definition}
A MAP topological group $X$ is said to have the {\rm sequentially bounded property} ({\rm $SB$-property} for short) if every functionally bounded sequence in $X^+$ is also functionally bounded in $X$.
\end{definition}

\begin{proposition} \label{p-Res}
Let $X$ be a MAP topological group. If $X$ has the $SB$-property, then every functionally bounded subset in $X^+$ is precompact in $X$.
\end{proposition}

\begin{proof}
Let $A$ be  a functionally bounded subset in $X^+$. Suppose for a contradiction that $A$ is not left (right) precompact in $X$. By Proposition \ref{p-UD}, there exists an infinite sequence
$\{ u_n\}_{n\in\NN}$ in $A$ which is left (right) uniformly discrete in $X$. Now Lemma \ref{l-FB} implies that $\{ u_n\}$ is not functionally bounded in $X$. But, as a subset of $A$, the sequence $\{ u_n\}$ is functionally bounded in $X^+$. This contradicts the $SB$-property. Thus $A$ is left and right precompact, and hence it is a precompact subset of $X$.
\end{proof}

\begin{corollary}
If  a MAP topological group $X$ has the $\SP$-property, then the following are equivalent:
\begin{enumerate}
\item[{\rm (i)}] $X$ has the $SB$-property.
\item[{\rm (ii)}] Every functionally bounded subset of $X^+$ is precompact in $X$.
\end{enumerate}
\end{corollary}

\begin{proof}
(i)$\Rightarrow$(ii) follows from Proposition \ref{p-Res}.

(ii)$\Rightarrow$(i) Let $A =\{ x_n\}_{n\in\NN}$ be a functionally bounded sequence in $X^+$. So $A$ is precompact in $X$. As $X$ has the $\SP$-property, the closure $B:=\mathrm{cl}_X (A)$ is compact in $X$. So $B$ is functionally bounded in $X$. Thus $A$ is also functionally bounded in $X$. Therefore $X$ has the $SB$-property.
\end{proof}






Now we obtain a sufficient condition  when a {\it complete} MAP group $X$ respects one of the properties from $\mathfrak{P}$. Recall that a topological space $Y$ is called a {\it $\mu$-space} if every functionally bounded subset of $Y$ is relatively compact.
\begin{proposition} \label{t-Res}
Let $X$ be a complete MAP topological group. If $X$ has the $SB$-property, then
\begin{enumerate}
\item[{\rm (i)}] $X$ respects each property from $\mathfrak{P}$;
\item[{\rm (ii)}] $X^+$ is  a $\mu$-space.
\end{enumerate}
\end{proposition}

\begin{proof}
Let $\mathcal{P}\in\mathfrak{P}$ be arbitrary.
Denote by $id:X\to X^+$ the identity map. Let $A$ be a subset of $X^+$ having the property $\mathcal{P}$. Then $A$ is functionally bounded in $X^+$ by Diagram \ref{diag}. By Lemma \ref{l-Boun}, the closure $B:=\mathrm{cl}_{X^+} (A)$ is also functionally bounded in $X^+$. Now Proposition \ref{p-Res} implies that $B$ is a closed precompact subset of $X$. So $B$ is a compact subset of $X$ by Fact \ref{f-Precompact}. Hence $B$ is compact in $X^+$. Since the restriction $id|_{B}$ of $id$ onto $B$ is a homeomorphism we obtain that  $A$ also has $\mathcal{P}$. Thus $X$ respects $\mathcal{P}$ that proves item (i). Now the partial case in this proof when $\mathcal{P}=\mathcal{FB}$ shows that $X^+$ is  a $\mu$-space.
\end{proof}

Taking into account \cite[Theorem 3.2]{HM} the next theorem  generalizes the equivalence of (a)-(d) in Theorem 3.3 of \cite{HM}:
\begin{theorem}
Let $X$ be a complete MAP topological group. Then the following are equivalent:
\begin{enumerate}
\item[{\rm (i)}] $X$ respects compactness and $X^+$ is a $\mu$-space.
\item[{\rm (ii)}] $X$ respects countable compactness and $X^+$ is a $\mu$-space.
\item[{\rm (iii)}] $X$ respects pseudocompactness and $X^+$ is a $\mu$-space.
\item[{\rm (iv)}] $X$ respects functional  boundedness and $X^+$ is a $\mu$-space.
\item[{\rm (v)}] $X$ has the $SB$-property.
\end{enumerate}
Hence, if (i)-(v) hold, then $X$ respects sequentiality and sequential compactness as well.
\end{theorem}

\begin{proof}
(i)-(iv)$\Rightarrow$(v) Let $A$ be a functionally bounded sequence in $X^+$. As $X^+$ is a $\mu$-space, the set $B:=\mathrm{cl}_{X^+} (A)$ is compact in $X^+$. By hypothesis on $X$ and Diagram \ref{diag}, $B$ is functionally bounded in $X$. Thus $A$ is a  functionally bounded subset of $X$ (see Lemma \ref{l-Boun}). Therefore $X$ has the $SB$-property.

(v) implies (i)-(iv) and the last assertion by Proposition \ref{t-Res}.
\end{proof}

For the properties from $\mathfrak{P}_0$ we complete Proposition \ref{t-Res} by the following:
\begin{proposition} \label{t-cRes}
Let $X$ be a MAP  topological group and $\mathcal{P}\in \mathfrak{P}_0$. If the completion $\overline{X}$ of $X$ respects $\mathcal{P}$, then $X$ also respects $\mathcal{P}$.
\end{proposition}

\begin{proof}
Denote by $i: X\to \overline{X}$ the natural embedding. Since $bX = b\overline{X}$ by Lemma \ref{l-BohrRep}, the identity map $i^+ :  X^+\to \overline{X}^+, i^+(x)=x,$ is also an embedding.
Let $A\in \mathcal{P}(X^+)$. Then $A=i^+(A) \in \mathcal{P}(\overline{X}^+)$. Since $\overline{X}$ respects $\mathcal{P}$, we have $A\in \mathcal{P}(\overline{X})$. Hence $A\in \mathcal{P}(X)$ as well. Thus  $X$  respects $\mathcal{P}$.
\end{proof}
Assume that a subcategory $\pmb{\mathbb{A}}$ of $\mathbf{MAP}$ is closed under taking of completions. Then Proposition \ref{t-cRes} shows that $\mathcal{P}\in \mathfrak{P}_0$ is a respected property in  $\pmb{\mathbb{A}}$ if and only if every {\it complete } $X\in \pmb{\mathbb{A}}$ respects $\mathcal{P}$.


The next proposition is an analogue of Proposition 2.1 in \cite{ReT}:
\begin{proposition} \label{pProdSchur}
Let $\{ X_i\}_{i\in I}$ be a non-empty family of MAP Abelian groups and $X:=\prod_{i\in I} X_i$.  Then
\begin{enumerate}
\item[{\rm (i)}] {\rm (\cite{ReT})} $X$ respects compactness  if and only if $X_i$  respects compactness for every $i\in I$.
\item[{\rm (ii)}] $X$ has the Schur property if and only if $X_i$ has the Schur property for every $i\in I$.
\item[{\rm (iii)}] $X$ respects functional boundedness  if and only if $X_i$  respects functional boundedness for every $i\in I$.
\item[{\rm (iv)}] $X$ has the $SB$-property if and only if  $X_i$  has the $SB$-property for every $i\in I$.
\item[{\rm (v)}] If $X$ respects pseudocompactness (respectively, countable compactness or sequential  compactness),  then $X_i$  respects  pseudocompactness (respectively, countable compactness or sequential  compactness) for every $i\in I$.
\end{enumerate}
\end{proposition}

\begin{proof}
Denote by $\pi_i$ the projection from $X$ onto $X_i$. Define $\nu_j :X_j \to X$ as follows:
\[
\nu_j (x_j):= (x_i), \mbox{ where } x_i = x_j \mbox{ if } i=j, \mbox{ and } x_i =0 \mbox{ otherwise}.
\]

(ii) Let $X$ have the Schur property. Fix arbitrarily an index $j\in I$. We have to show that $X_j$  has the Schur property. Let a sequence $\{ x^n_j\}_{n\in\omega}$ in $X_j$ converge to zero in $X_j^+$. For every $n\in\NN$ we set $\mathbf{x}_n =\nu_j (x^n_j)$. Then for each $\chi\in \widehat{X}$, we have $(\chi, \mathbf{x}_n)=(\chi|_{X_j}, x^n_j)\to 1$. Thus $\mathbf{x}_n \to 0$ in $\sigma(X,\widehat{X})$. Since $X$ has the Schur property, $\mathbf{x}_n \to 0$ in the product topology on $X$. Hence $x^n_j = \pi_j (\mathbf{x}_n) \to 0$ in $X_j$. Thus $X_j$ has the Schur property.

Conversely, let  $X_i$ have the Schur property for every $i\in I$, and let $\{ \mathbf{x}_n\}_{n\in\omega}$ be a null sequence in $X^+$. By Fact \ref{f-Bohr}, $\pi_i (\mathbf{x}_n) \to 0$ in $X_i^+$ for every $i\in I$. Since $X_i$ has the Schur property, $\pi_i (\mathbf{x}_n) \to 0$ in $X_i$ for every $i\in I$. Now, for every $U\in \mathcal{N}(X)$ of the form
\[
U= U_{i_1} \times \dots \times U_{i_t} \times \prod_{i\in I\setminus \{i_1,\dots, i_t\}} X_i , \mbox{ where } U_{i_k}\in \mathcal{N}(X_{i_k}) \mbox{ and } 1\leq k\leq t,
\]
choose $m\in \NN$ such that $\pi_{i_k}(\mathbf{x}_n)\in U_{i_k}$ for every $n\geq m$ and each $1\leq k\leq t$. Then $\mathbf{x}_n \in U$ for every $n\geq m$. This means that $\mathbf{x}_n \to 0$ in $X$. Thus  $X$ has the Schur property.

(iii),(iv) Let  $X$ respect functional boundedness (respectively, $X$  have the $SB$-property). Fix arbitrarily an index $j\in I$. We have to show that $X_j$ respects functional boundedness (respectively, $X_j$ has  the $SB$-property). Let $A_j$ be a functionally bounded subset (respectively, sequence) in $X_j^+$. Set $A := \nu_j(A_j)$.
Then $A$ is functionally bounded in $X_j^+ \times \left(\prod_{i\in I, i\not=j} X_i\right)^+ = X^+$  by Lemma  \ref{l-Boun} and  Fact \ref{f-Bohr}. By assumption, $A$ is  functionally bounded in $X$. Hence $\pi_j (A)=A_j$ is  functionally bounded in $X_j$. Thus $X_j$ respects  functional boundedness  (respectively, $X_j$  has the $SB$-property).

Conversely, let  $X_i$  respect  functional boundedness  (respectively,   $X_i$  have the $SB$-property) for every $i\in I$, and let $A$ be a functionally bounded subset (respectively, sequence) in $X^+$. By   Lemma \ref{l-Boun} and  Fact \ref{f-Bohr}, the set $\pi_i (A)$ is  functionally bounded in $X^+_i$, and hence in $X_i$, for every $i\in I$. Now Theorem 2.2 of \cite{Tka88} implies that the set $E:= \prod_{i\in I} \pi_i (A)$ is  functionally bounded in $X$. Since $A\subseteq E$, $A$ is also  functionally bounded in $X$ by Lemma \ref{l-Boun}. Thus $X$ respects  functional boundedness  (respectively, $X$ has  the $SB$-property).



(v)  We consider only the pseudocompact case. Fix arbitrarily an index $j\in I$. We have to show that $X_j$ respects  pseudocompactness. Taking into account that each topological group is Tychonoff, we note that   the image of a pseudocompact  subset under a continuous homomorphism is pseudocompact \cite[3.10.24]{Eng}. Now let  $A$ be a pseudocompact subset in $X_j^+$. Then $\nu_j (A)$ is a pseudocompact subset of $X^+ = X_j^+ \times  \left(\prod_{i\in I, i\not=j} X_i\right)^+ $. Since $X$  respects  pseudocompactness, $A$ is pseudocompact in $X$. Now $\pi_j(\nu_j(A))=A$ is  pseudocompact in $X_j$.
\end{proof}

In the next proposition we consider heredity of respected properties. Note that,  for the Schur property, item (i) was noticed in \cite[\S 2]{MT}, and, for respect compactness, this item is folklore (for dually embedded subgroups it is proved in \cite[2.4]{ReT}):

\begin{proposition} \label{pSubSchur}
Let $H$ be a subgroup of a MAP Abelian group $X$ and $\mathcal{P}\in \mathfrak{P}_0$.
\begin{enumerate}
\item[{\rm (i)}] If $X$ respects $\mathcal{P}$, then $H$ respects $\mathcal{P}$ as well.
\item[{\rm (ii)}]  If $H$ is open and   respects $\mathcal{P}$, then $X$   respects $\mathcal{P}$.
\end{enumerate}
\end{proposition}

\begin{proof}
(i) Let $K\in \mathcal{P}(H^+)$. Denote by $i: H^+ \to X^+$ the identity map.
By Lemma \ref{lBohr}(3), $i$ is continuous. So $K=i(K)\in \mathcal{P}(X^+)$  (note that Hausdorff topological groups are Tychonoff). Hence $K\in \mathcal{P}(X)$. Thus $K\in \mathcal{P}(H)$. Therefore $H$  respects $\mathcal{P}$.

(ii) Let $K\in \mathcal{P}(X^+)$. Since $K$ is also functionally bounded in $X^+$, $K$ is contained in a finite union of cosets of $H$ by Lemma \ref{l-FbMAP}.  Since $H$ is open, Fact \ref{fOpen} and Lemma \ref{lBohr}(4) imply that  $H$ is a closed subgroup of $X^+$. Hence, without loss of generality, we can assume that $K\subseteq H^+$ and $K\in \mathcal{P}(H^+)$. As  $H$  respects  $\mathcal{P}$, $K\in \mathcal{P}(H)$, and hence $K\in \mathcal{P}(X)$. Thus $X$ respects  $\mathcal{P}$.
\end{proof}

For  functional boundedness and the $SB$-property the situation is more complicated. Recall that a subspace $E$ of a topological space $Y$ is called {\it $C$-embedded} (respectively, {\it $C^\ast$-embedded}) in $Y$ if every continuous function (respectively, every bounded continuous function) on $E$ can be extended to a continuous function (respectively, to a bounded continuous function) on $Y$.
\begin{proposition} \label{pSubFB}
Let $H$ be a subgroup of a  MAP Abelian group $X$.
\begin{enumerate}
\item[{\rm (i)}] If $X$ respects  functional boundedness and $H$ is $C$-embedded in $X$, then $H$ also respects  functional boundedness.
\item[{\rm (ii)}]  Let $H$ be open and respect  functional boundedness. If $H^+$  is $C$-embedded in $X^+$, then $X$  respects  functional boundedness.
\item[{\rm (iii)}] If $X$ has  the $SB$-property and $H$ is $C$-embedded in $X$, then $H$ also has   the $SB$-property.
\item[{\rm (iv)}]  Let $H$ be open and  have  the $SB$-property. If $H^+$  is $C$-embedded in $X^+$, then $X$   has   the $SB$-property.
\end{enumerate}
\end{proposition}

\begin{proof}
(i),(iii) Let $A$ be a functionally bounded subset (respectively, sequence) in $H^+$. As $\sigma(X,\widehat{X})|_H \leq  \sigma(H,\widehat{H})$ by Lemma \ref{lBohr}, $A$ is also   functionally bounded in $X^+$. Hence, by assumption, $A$ is    functionally bounded in $X$. Since $H$ is $C$-embedded, $A$ is functionally bounded in $H$. Thus $H$  respects  functional boundedness (respectively, $H$ has   the $SB$-property).

(ii),(iv) Let  $A$ be a functionally bounded subset (respectively, sequence) in $X^+$. By Lemma \ref{l-FbMAP}, $A$ is contained in a finite union of cosets of $H$. Hence, by Lemma \ref{l-Boun}(1), without loss of generality we can assume that $A\subseteq H$.

Fact \ref{fOpen} and Lemma \ref{lBohr} imply that $H^+$ is a closed subgroup of $X^+$. Since  $H^+$  is $C$-embedded in $X^+$, $A$ is functionally bounded in $H^+$. So, by assumption, $A$ is functionally bounded in $H$. Thus $A$ is functionally bounded in $X$. Therefore, $X$  respects  functional boundedness (respectively, $X$ has   the $SB$-property).
\end{proof}
We do not know whether it is possible to weaken conditions on $H$ and $H^+$ in Proposition \ref{pSubFB}. Note that by Lemma \ref{lBohr}(5), a subgroup $H$ of a MAP Abelian group $X$ is dually closed and dually embedded if and only if $H^+$ is a closed subgroup of $X^+$. So the next general question is of interest:
\begin{problem} \label{problemBohr}
Let $X$ be a MAP Abelian group. Characterize those dually closed and dually embedded subgroups $H$ of $X$ such that $H^+$ is $C$-embedded in $X^+$.
\end{problem}
If $G$ is a LCA group, it is well-known that every closed subgroup $S$ of $G$ is dually closed and dually embedded in $G$. Moreover, $S^+$ is automatically $C$-embedded in $G^+$ by Theorem 5.6 of \cite{CHTr}. It is known (see \cite[8.3 and 8.6]{Ban}) that every closed subgroup of a nuclear group is dually closed and dually embedded. This justifies the next question which is a partial case of \cite[Problem 7.5]{CHTr}:
\begin{problem}
Let $H$ be a closed subgroup of a nuclear group $X$. Is $H^+$ $C$-embedded in $X^+$?
\end{problem}
It is also interesting to consider Problem \ref{problemBohr} for Schwartz groups. In the next question we reformulate and extend Problem \ref{problemBohr}:
\begin{problem}
Characterize $C$-embedded and $C^\ast$-embedded closed subgroups of precompact (Abelian) groups.
\end{problem}

\section{Proof of Theorem \ref{t11}} \label{secF}

We prove  Theorem \ref{t11} by a reduction of the general case to several simpler ones using the structure theory of LCA groups and the results of the previous section.
We need some lemmas.

It is well-known that the space $c_0$ does not have the Schur property. Below we generalize this fact.
\begin{lemma} \label{l-c_0}
The space $c_0=\mathfrak{F}_0(\mathbb{R})$ does not respect any property from $\mathfrak{P}$.
\end{lemma}

\begin{proof}
By Proposition \ref{p-NoRes} it is enough to find a convergent sequence in $X^+$ which is not functionally bounded in $X$.
Set $e_0 :=0\in c_0$ and $e_n :=\nu_n(1)$ for every $n\in\NN$.
We claim  that $e_n \to e_0$  in $c_0^+$. Indeed, for every $\mathbf{s}=(s_i)_{i\in\NN} \in \ell_1 =c_0^\wedge$, we have $(\mathbf{s}, e_n)=s_n \to 0$. Thus $e_n \to e_0$ in $c_0^+$.

On the other hand, the set  $K:= \{ e_n\}_{n\in\omega}$  is not  functionally bounded in $c_0$. Indeed, since $\| e_n - e_m \| \geq 1$ for all distinct $n,m\in\omega$, the set  $K$ is uniformly discrete in $c_0$. By Lemma \ref{l-FB}, the set $K$ is not  functionally bounded in $c_0$.
\end{proof}

Next proposition generalizes Lemma \ref{l-c_0}.
\begin{proposition}
If $L$ is a nontrivial real LCS, then $\mathfrak{F}_0(L)$ does not respect any property from $\mathfrak{P}$.
\end{proposition}

\begin{proof}
Since $L\not= \{ 0\}$ we can represent $L$ in the form $L=\mathbb{R}\oplus L_0$, where $L_0$ is a closed subspace of $L$ \cite[\S 20.5(5)]{Kothe}. So $\mathfrak{F}_0(L)\cong c_0 \oplus \mathfrak{F}_0(L_0)$ by Fact \ref{p11}. Now the assertion follows from Proposition \ref{pProdSchur} and Lemma \ref{l-c_0}.
\end{proof}

The connected component of a topological group $X$ we denote by $C(X)$.

\begin{lemma} \label{l-F_0}
Let $X$ be a LCA group containing an open compact subgroup. If $C(X)\not= \{ 0\}$, then $\mathfrak{F}_0(X)$  does not respect  any property from $\mathfrak{P}$.
\end{lemma}

\begin{proof}
We use  Proposition \ref{p-NoRes}.
Take arbitrarily a non-zero $x\in C(X)$. Set $\mathbf{x}_0 :=0\in \mathfrak{F}_0(X)$ and $\mathbf{x}_n :=\nu_n(x)$ for every $n\in\NN$. If a symmetric $U\in \mathcal{N}(X)$ is such that $x\not\in U$, then $K:= \{ \mathbf{x}_n\}_{n\in\omega}$ is $(U^\NN \cap c_0(X))$-separated. Thus $K$ is not functionally bounded in $\mathfrak{F}_0(X)$ by Lemma \ref{l-FB}.
So to prove the lemma it is enough to show that $\mathbf{x}_n \to \mathbf{x}_0$  in $\mathfrak{F}_0(X)^+$.

Let $\mathbf{g} =(g_n)_{n\in\NN} \in \widehat{\mathfrak{F}_0(X)}$. By Fact \ref{f3}(2),  there exists an open subgroup $H$ of $X$ and a natural number $m$ such that $g_n \in H^\perp$ for every $n\geq m$. Noting that $C(X)\subseteq H$, we obtain that $(\mathbf{g}, \mathbf{x}_n)=(g_n, x)=1$ for every $n\geq m$. Thus $\mathbf{x}_n \to \mathbf{x}_0$ in $\mathfrak{F}_0(X)^+$. 
\end{proof}

In spite of the next two lemmas  follows from  \cite[Theorem]{BaMP2} and Proposition \ref{pSubT}(ii), we give here their independent proofs.
\begin{lemma} \label{l-FbTD}
Let $X$ be a totally disconnected LCA group. If $E$ is a functionally bounded subset in $\mathfrak{F}_0(X)^+$,  then $E$ is functionally bounded also in $\mathfrak{F}_0(X)$. Moreover, $\mathrm{cl}_{\mathfrak{F}_0(X)} (E)$ is a compact subset of $\mathfrak{F}_0(X)$ and $\mathfrak{F}_0(X)^+$.
\end{lemma}

\begin{proof}
We show first that the sequence $\{ \pi_n (E)\}_{n\in\NN}$ is a null-sequence in $X$. Suppose for a contradiction that $\{ \pi_n (E)\}_{n\in\NN}$ is not a null-sequence. Then, by Lemma \ref{lNull2}, $E$ has a uniformly discrete sequence $\{ \mathbf{b}_n\}_{n\in\NN}$. Since $X$ is totally disconnected it has a subgroup topology \cite[7.7]{HR1}. Hence $\mathfrak{F}_0(X)$ has a subgroup topology by  Proposition \ref{pSubT}(i). So there exists an open subgroup $H$ of $X$ such that $\{ \mathbf{b}_n\}$ is $(H^\NN \cap c_0(X))$-separated. Thus $\{ \mathbf{b}_n\}$, and hence $E$,  is not contained in a finite union of cosets of $H^\NN \cap c_0(X)$. As $H^\NN \cap c_0(X)$ is an open subgroup of $\mathfrak{F}_0(X)$ by Fact \ref{p11}(2),  Lemma \ref{l-FbMAP} implies that $E$ is not functionally bounded in $\mathfrak{F}_0(X)^+$, a contradiction.

Since the set  $\pi_n (E)$ is functionally bounded in $X^+$  for every $n\in\NN$ by Lemma \ref{l-Boun},  Theorem 4.2 of \cite{Tri91} implies that $\pi_n (E)$ is also functionally bounded in $X$ and $K_n :=\mathrm{cl}_X (\pi_n (E))$ is compact in $X$. Now Lemma \ref{lNull1} implies that $\{ K_n \}_{n\in\NN}$ is a null-sequence of compact subsets in $X$. Set $K:= \prod_{n\in\NN} K_n$. Fact \ref{p11}(3) yields that $K$ is a compact subset of $\mathfrak{F}_0(X)$. Clearly, $\mathrm{cl}_{\mathfrak{F}_0(X)} (E) \subseteq K$. Hence $\mathrm{cl}_{\mathfrak{F}_0(X)} (E)$ is compact in $\mathfrak{F}_0(X)$. Thus $E$ is functionally bounded in $\mathfrak{F}_0(X)$  by Lemma \ref{l-Boun}.
\end{proof}

\begin{lemma} \label{l-FbPC}
Let $X$ be a totally disconnected LCA group. If $E$ is a pseudocompact (respectively,  countably compact or sequentially compact) subset of $\mathfrak{F}_0(X)^+$,  then $E$ is pseudocompact  (respectively,  countably compact or sequentially compact) also in $\mathfrak{F}_0(X)$.
\end{lemma}

\begin{proof}
Since $E$ is  pseudocompact  (respectively,  countably compact or sequentially compact) in $\mathfrak{F}_0(X)^+$, then it is also functionally bounded in $\mathfrak{F}_0(X)^+$. Hence $K:= \mathrm{cl}_{\mathfrak{F}_0(X)} (E)$ is compact in $\mathfrak{F}_0(X)$  by Lemma \ref{l-FbTD}. So the topologies of $\mathfrak{F}_0(X)$ and $\mathfrak{F}_0(X)^+$ coincide on $K$. In particular, $\mathfrak{u}_0|_E =\mathfrak{u}_0^+|_E$. Thus $E$ is pseudocompact  (respectively,  countably compact or sequentially compact)  in $\mathfrak{F}_0(X)$.
\end{proof}


The class of nuclear groups was introduced by Banaszczyk in \cite[Chapter 3]{Ban}. The concept of a Schwartz topological Abelian group appeared in \cite{ACDT}. This notion generalizes the well-known notion of a Schwartz locally convex space. Recall that an Abelian Hausdorff group $X$ is called a {\it Schwartz group} if for every $U\in \mathcal{N}(X)$ there exists a  sequence $\{ F_n\}_{n\in\NN}$ of finite subsets of $X$ such that the intersection $\bigcap_{n\in\NN} ( (1/n)U +F_n)$ is a neighborhood of $0$, where $(1/n)U:= \{ x\in X: kx\in U \; \forall 1\leq k\leq n\}$.
All nuclear groups as well as the Pontryagin dual groups of  metrizable  Abelian groups are Schwartz groups (see \cite{ACDT}).

Now we are in position to prove Theorem \ref{t11}.

{\large Proof of Theorem  {\rm \ref{t11}}}.
(i)$\Rightarrow$(ii). Let $X$ be totally disconnected. Then $X$ has a subgroup topology by \cite[7.7]{HR1}. Hence $\mathfrak{F}_0(X)$ also has a subgroup topology by  Proposition \ref{pSubT}(i). Thus $\mathfrak{F}_0(X)$ embeds into a direct product of discrete Abelian groups (see, for example, Proposition 2.2 in \cite{AG}).

(ii)$\Rightarrow$(iii). Every LCA group is nuclear \cite[7.10]{Ban}. Thus $\mathfrak{F}_0(X)$ is nuclear by \cite[7.5 and 7.6]{Ban}.

(iii)$\Rightarrow$(iv). Every  nuclear group is a  Schwartz groups  by  \cite{ACDT}.

(iv)$\Rightarrow$(v) follows from Theorem 4.4 of \cite{Aus2} since $\mathfrak{F}_0(X)$ is locally quasi-convex by Fact \ref{fTopPropF}(i).


(v)$\Rightarrow$(vi) is clear (see also \cite{BMPT}).

(vi)$\Rightarrow$(i), (vii)$\Rightarrow$(i), (viii)$\Rightarrow$(i), (ix)$\Rightarrow$(i), (x)$\Rightarrow$(i) Let $X\cong \mathbb{R}^n \times X_0$, where $n\geq 0$ and $X_0$ has an open compact subgroup $K$ \cite[24.30]{HR1}.

By Fact \ref{f3}, $\mathfrak{F}_0 (X)\cong c_0^n \times \mathfrak{F}_0 (X_0)$. Since $c_0$ does not respect sequentiality (respectively, countable compactness, sequential compactness, pseudocompactness and functional boundedness) by Lemma \ref{l-c_0},  we obtain that  $n=0$ by Proposition \ref{pProdSchur}. Hence $X=X_0$ has an open compact subgroup $K$. Now Lemma \ref{l-F_0} implies that $C(X)$ is trivial. Thus $X$ is totally disconnected.

The implications (i)$\Rightarrow$(vii), (i)$\Rightarrow$(viii), (i)$\Rightarrow$(ix)  and (i)$\Rightarrow$(x) and the last assertion of the theorem follow from Lemmas \ref{l-FbTD} and \ref{l-FbPC}.
$\Box$

Recall \cite{DiM} that a Hausdorff topological group $X$ is called {\it locally $q$-minimal} if there exists a neighborhood $V$ of the identity $e_X$ such that whenever $H$ is a Hausdorff group and $p:X\to H$ is a continuous surjective homomorphism such that $p(V)$ is a neighborhood of $e_H$, then $p$ is open. The next corollary of Theorem \ref{t11} and \cite[Theorem 1.3]{DiM} was pointed out to the author by D.~Dikranjan:
\begin{corollary}
If $X$ is a totally disconnected non-discrete LCA group, then $\mathfrak{F}_0 (X)$ is not locally $q$-minimal.
\end{corollary}

\begin{proof}
By Theorem \ref{t11}, the group $\mathfrak{F}_0 (X)$ embeds onto a subgroup $H$ of a product of locally compact groups. Since $\mathfrak{F}_0 (X)$ is complete by Fact \ref{fTopPropF}(i), $H$ is closed. Fact \ref{fTopPropF}(iii) implies that $H$ is not locally compact. So $H$, and hence $\mathfrak{F}_0 (X)$, is not locally $q$-minimal by \cite[Theorem 1.3]{DiM}.
\end{proof}
We do not know whether $\mathfrak{F}_0 (X)$ is  locally $q$-minimal for a connected LCA group $X$.

\begin{remark} \label{r-RP} {\em
Note that the  Glicksberg property has been studied beyond LCA groups by many authors, see \cite{Aus2, BaMP, BaMP2, GaH, ReT3, WuR}. The Schur property in Abelian topological groups was intensively studied by Mart\'{\i}n-Peinador and Tarieladze in \cite{MT} (see also  \cite{Ga7, HGM}). Many other results concerning  preservation and respectness of topological  properties under the Bohr functor see \cite{BMPT}. A complete characterization of Banach spaces which have the Schur property is given in \cite{HGM}. }
\end{remark}

\section{Proof of Theorem \ref{t12}} \label{sec2}

Let $X$ be a metrizable Abelian group and $\rho$ be an invariant metric for $X$.  For every $n\in\NN$, set $U_n^X =U_n :=\{ x\in X: \rho(x,0)<\frac{1}{n}\}$. Note that the equality
\[
d(\mathbf{x},\mathbf{y}):= \sup_{n\in\NN} \rho(x_n, y_n), \quad \forall \mathbf{x}=(x_n), \mathbf{y}=(y_n)\in c_0(X),
\]
defines an invariant metric for $\mathfrak{F}_0(X)$ \cite{DMPT}. Analogously, for every $m\in\NN$, we consider $X^m$ with the metric
\[
d_m ((x_1,\dots,x_m),(y_1,\dots,y_m)) := \max\{ \rho(x_1, y_1),\dots, \rho(x_m, y_m)\}.
\]
So $U_n^{X^m} = U_n^m$ for every $n,m\in\NN$.

Recall that a subset $A$ of $X$ is called {\it $\varepsilon$-dense} in $X$ if for every $x\in X$ there exists $a\in A$ such that $\rho(x,a)<\varepsilon$.
For every  $n,m\in\NN$, we denote by $B_X (U_m, n)$ the set of all $z\in X$ such that the set $F_n (z):=\{ -nz,\dots, -z, 0, z,\dots, nz\}$ is $\frac{1}{m}$-dense in $X$.

\begin{lemma} \label{l-Edense}
Let $X$ be a compact metrizable Abelian group with an invariant metric $\rho$. Then, for every  $n,m\in\NN$, the set $B_X (U_m, n)$ is open in $X$.
\end{lemma}

\begin{proof}
If $B_X(U_m, n)=\emptyset$, it is open. Assume now that $B_X(U_m, n)$ is not empty. For every $x, z\in X$, we define
\[
f(x,z):= \min\{ \rho(x,kz): k\in\{  -n,\dots, -1, 0, 1,\dots, n\} \} = \rho(x, F_n(z)).
\]
Then $f(x,z)$ is continuous on $X\times X$. Clearly, $F_n (z)$ is $\frac{1}{m}$-dense in $X$ if and only if $f(x, z)<\frac{1}{m}$ for every $x\in X$.

Fix $z_0 \in B_X(U_m, n)$. Since $X$ is compact, we can set $\varepsilon_0 := \max\{ f(x,z_0): x\in X\}$. Then $\varepsilon_0 < \frac{1}{m}$.  As $f(x,z)$ is also uniformly continuous,  there is $\delta \in \NN$ such that
\[
| f(x',z') - f(x,z)| < \frac{1}{2} \left(\frac{1}{m} - \varepsilon_0\right),
\]
for every $(x',z'), (x,z) \in X\times X$ with $(x',z')- (x,z) \in U_\delta \times U_\delta $. In particular, for each $z\in z_0 + U_\delta$, we have
\[
| f(x,z) - f(x,z_0)| < \frac{1}{2} \left(\frac{1}{m} - \varepsilon_0\right), \ \forall x\in X.
\]
So, for every $z\in z_0 + U_\delta$ and each $x\in X$, we obtain
\[
f(x,z) < f(x,z_0) + \frac{1}{2} \left(\frac{1}{m} - \varepsilon_0\right) \leq \varepsilon_0 + \frac{1}{2} \left(\frac{1}{m} - \varepsilon_0\right) =\frac{1}{2} \left(\frac{1}{m} + \varepsilon_0\right)< \frac{1}{m}.
\]
This means that $F_n (z)$ is $\frac{1}{m}$-dense in $X$ for every $z\in z_0 + U_\delta$. Therefore, $B_X(U_m, n)$ is  open in $X$.
\end{proof}

Let $X$ be a compact {\it connected} metrizable Abelian group and  $\mathrm{m}_X$ be the normalized Haar measure on $X$.
Denote by $S(X)$ the set of all elements $x$ of $X$ such that the cyclic subgroup $\langle x \rangle$ generated by $x$ is dense in $X$.
By \cite[25.27]{HR1}, $S(X)$ is measurable and
\begin{equation} \label{eMon1}
\mathrm{m}_X (S(X))=1.
\end{equation}


\begin{lemma} \label{lMon}
Let $X$ be a compact connected metrizable Abelian group and let $z$ be an element of $S(X)$. Then
\begin{enumerate}
\item[{\rm (1)}] $n z\in S(X)$ for every $n\in\NN$.
\item[{\rm (2)}] For every $U\in \mathcal{N}(X)$ there exists $n_U\in\NN$ which satisfies the next condition: for each $x\in X$ there is an integer number $n$ with $|n|\leq n_U$ such that $x-nz\in U$.
\end{enumerate}
\end{lemma}

\begin{proof}
(1) Set $X_n :=\mathrm{cl}(\langle n z \rangle)$. We have to show that $X_n =X$. Suppose for a contradiction that $X_n \not= X$. Then $\langle q(z) \rangle$ is dense in $X/X_n$, where $q: X\to X/X_n$ is the quotient map. Since the group $\langle q(z) \rangle$ is finite, we obtain that the connected group $X/X_n$ is non-trivial and finite, a contradiction.

(2) Take a symmetric $V\in \mathcal{N}(X)$ such that $V+V \subseteq U$. Since $X$ is compact, there is a finite subset $\{ y_l \}_{l=1}^m$ of $X$ such that  $\{ y_l +V\}_{l=1}^m$ is a cover of $X$. For every $1\leq l\leq m$, choose $n_l \in\ZZ$ such that $n_l z \in y_l +V$ and set $n_U := \max\{ |n_1|,\dots,|n_m|\}$.

Let $x\in X$ be arbitrary. Take $1\leq l\leq m$ such that $x\in y_l +V$. Then $x- n_l z \in V-V\subseteq U$ and $|n_l|\leq n_U$. Thus $n_U$ is as desired.
\end{proof}

Let $X$ be a compact connected metrizable Abelian group.
For every   $m,n\in\NN$, set
\[
A_X(U_m,n) := B_X(U_m, n) \cap S(X).  
\]
Then $A_X(U_m,n) \subseteq A_X(U_m,n+1)$, for every $n,m\in\NN$, and Lemma \ref{lMon}(2) implies that
\begin{equation} \label{eMon2}
S(X)=\bigcup_{n\in\NN} A_X(U_m,n) , \quad \forall m\in \NN.
\end{equation}
By Lemma \ref{l-Edense}, $A_X(U_m,n)$ is a measurable subset of $X$ for every   $m,n\in\NN$.



\begin{proposition} \label{pMon}
Let $X$ be a compact connected metrizable Abelian group. Then the group $\mathfrak{F}_0 (X)$ is monothetic.
\end{proposition}

\begin{proof}
Let $\rho$ be an invariant metric for $X$. We separate the proof into two steps.

{\it Step} 1. Let us show first that for every $m\in \NN$ there exists a natural number $n_m$ and a compact subset $K_m$ of the direct product $X^m$ such that the sequences  $\{ n_m\}$ and $\{ K_m\}$  satisfy the following conditions:
\begin{enumerate}
\item[{\rm (i)}] $1=n_0 <n_1 <n_2< \dots$;
\item[{\rm (ii)}] $K_m \subseteq A_{X^m} \left( \left( U_{2^m n_{m-1}}\right)^m , \ n_m\right) =A_{X^m} \left(  U_{2^m n_{m-1}}^{X^m} , \ n_m\right)$ for every $m\in \NN$;
\item[{\rm (iii)}] $K_{m+1} |_{X^{m}} \subseteq K_{m}$ for every $m\in \NN$;
\item[{\rm (iv)}] $\pi_m (K_m) \subseteq U_{2^m n_{m-1}}$ for every $m\in \NN$;
\item[{\rm (v)}] $\mathrm{m}_{X^m} (K_m) >0$ for every $m\in \NN$.
\end{enumerate}
We build such sequences $\{ n_m\}$ and $\{ K_m\}$ by induction. For $m=1$, by (\ref{eMon1}) and (\ref{eMon2}), we can choose $n_1 >1$ such that
\[
\mathrm{m}_X \left( A_X(U_2, n_1) \cap U_2 \right)>0.
\]
Take arbitrarily a compact subset $K_1$ of $A_X(U_2, n_1) \cap U_2$ such that  $\mathrm{m}_X (K_1)>0$. Clearly,  $n_1$ and $K_1$ satisfy (i)-(v).

Suppose we have already built $1=n_0 <n_1 < \dots< n_m$ and $K_1,\dots , K_m$ which satisfy  (i)-(v). Since $\mathrm{m}_{X^{m+1}} \left( K_m \times U_{2^{m+1} n_{m}}\right) >0$,  (\ref{eMon1}) and (\ref{eMon2}) imply that there is $n_{m+1} > n_m$ such that
\[
\mathrm{m}_{X^{m+1}} (E_{m+1})>0,
\]
where  $E_{m+1} := A_{X^{m+1}} \left(  \left( U_{2^{m+1} n_{m}}\right)^{m+1} , \ n_{m+1} \right) \cap (K_m \times  U_{2^{m+1} n_{m}})$. Now take  arbitrarily a compact subset $K_{m+1}$ of $E_{m+1}$ such that  $\mathrm{m}_{X^{n+1}} (K_{m+1})>0$. Clearly, $n_1,\dots, n_{m+1}$ and $K_1,\dots, K_{m+1}$  satisfy (i)-(v).

{\it Step} 2. For every $m\in\NN$, set $K'_m := K_m \times X^{\NN \setminus \{ 1,\dots,m\}}$. Then, by (iii) and (v), $\{ K'_m\}$ is a decreasing sequence of non-empty compact subsets of $X^\NN$. Take arbitrarily $\mathbf{z}=(z_m) \in \cap_{m\in\NN} K'_m$. Then $z_m \in \pi_m (K_m)$, and hence $\mathbf{z}\in c_0(X)$ by (iv).
To prove the proposition it is enough to show that  $\langle \mathbf{z} \rangle$ is dense in $\mathfrak{F}_0(X)$.

Let $\varepsilon>0$ and $\mathbf{x}=(x_n)\in c_0(X)$. Choose $m\in\NN$ such that
\begin{equation} \label{eMon3}
\frac{1}{2^m} + \sup_{n>m} \rho(x_n, 0) <\frac{\varepsilon}{2}.
\end{equation}
Note that $\mathbf{z}_m :=(z_1,\dots, z_m)\in K_m$. Hence, by (ii), for $\mathbf{x}_m :=(x_1,\dots, x_m)$ there exists an integer number $k$  such that $|k|\leq n_m$ and
\begin{equation} \label{eMon4}
\mathbf{x}_m - k\cdot \mathbf{z}_m =(x_1- k z_1, \dots, x_m- k z_m) \in \left( U_{2^m n_{m-1}}\right)^m .
\end{equation}
It remains to prove that $d(\mathbf{x},k \mathbf{z})<\varepsilon$.

For $n=1,\dots, m$, by (\ref{eMon3}) and (\ref{eMon4}), we have
\begin{equation} \label{eMon5}
\rho(x_n, kz_n)< \frac{1}{2^m n_{m-1}} <\frac{\varepsilon}{2}.
\end{equation}
Let $n>m$. Since $|k|\leq n_m$, $z_n \in \pi_n (K_n)$ and $\rho$ is invariant, we obtain
\begin{equation} \label{eMon6}
\rho(x_n, kz_n)          \leq \rho(x_n,0) +\rho(0, kz_n) \leq \rho(x_n, 0)+ |k| \rho(0, z_n)
                         \stackrel{\mathrm{(i),(iv)}}{\leq} \rho(x_n, 0) +n_m\cdot \frac{1}{2^{m+1} \cdot n_m} \stackrel{(\ref{eMon3})}{<} \frac{\varepsilon}{2}.
\end{equation}
Now (\ref{eMon5}) and (\ref{eMon6}) imply that $d(\mathbf{x},k \mathbf{z})<\varepsilon$. Hence $\langle \mathbf{z}\rangle$ is dense in $\mathfrak{F}_0 (X)$. Thus $\mathfrak{F}_0 (X)$ is monothetic.
\end{proof}

Following \cite{CR66} we say that a topological group $X$ has {\it property UB} if each real-valued uniformly continuous function on $X$ is bounded.
In the next proposition we generalize Example 4.2 in \cite{CR66}.
\begin{proposition} \label{pUB}
Let $X$ be a connected compact Abelian group. Then the groups $(X^\NN,\mathfrak{u})$ and $\mathfrak{F}_0(X)$ have the property UB.
\end{proposition}

\begin{proof}
Let $f(\mathbf{x})$ be a real-valued uniformly continuous function on $(X^\NN,\mathfrak{u})$ (respectively, on $\mathfrak{F}_0(X)$). Take $U\in \mathcal{N}(X)$ such that $\mbox{ whenever } \mathbf{x}, \mathbf{z} \in X^\NN \mbox{ and } \mathbf{x} -\mathbf{z} \in U^\NN \; (\mbox{respectively}, \mathbf{x}, \mathbf{z} \in c_0(X) \mbox{ and } \mathbf{x} -\mathbf{z} \in U^\NN\cap c_0(X))$ we have
\begin{equation} \label{eUB1}
| f(\mathbf{x}) -f(\mathbf{z})|<1.
\end{equation}

Since $X$ is connected we have
\[
X=\bigcup_{n\in\NN} (n)U, \mbox{ where } (n)U:= \underbrace{U+\dots +U}_{n}.
\]
As $X$ is also compact, there is a natural number $m$ such that $X=(m)U$.

Fix $\mathbf{x}:=(x_i)_{i\in\NN} \in X^\NN$ (respectively, $\mathbf{x}\in \mathfrak{F}_0(X)$). For every $i\in\NN$ set $r_i :=\min\{ l\in\NN : x_i\in (l)U\}$. So $r_i\leq m$ for every $i\in\NN$ and, if $\mathbf{x}\in \mathfrak{F}_0(X)$, then $r_i =1$ for all sufficiently large $i$. Hence for every $i\in\NN$ there are nonzero elements $u_{1,i},\dots, u_{r_i, i} \in U$ such that
\[
x_i = u_{1,i}+ \dots + u_{r_i, i}.
\]
Set $s:=\max\{ r_i : i\in\NN\}$. So $s\leq m$. For every $1\leq l\leq s$ we set
\[
y^l_i := \left\{
\begin{array}{cc}
  u_{i,l}, & \mbox{ if } l\leq r_i\\
  0, & \mbox{ if } l> r_i
\end{array}
\right. ,
\]
and put $\mathbf{y}_l := (y^l_i)_{i\in\NN}$. Clearly, for every $1\leq l\leq s$ we have $\mathbf{y}_l \in U^\NN$. In the case $\mathbf{x}\in \mathfrak{F}_0(X)$ we have  $\mathbf{y}_1 \in U^\NN \cap c_0(X)$ and $\mathbf{y}_l \in U^\NN \cap X^{(\NN)}$ for every $2\leq l\leq s$. So
\begin{equation} \label{eUB}
\mathbf{x} = \mathbf{y}_1 +\dots +\mathbf{y}_s, \mbox{ where } \mathbf{y}_l\in U^\NN \; \left(\mbox{respectively},  \mathbf{y}_l \in U^\NN \cap\mathfrak{F}_0(X)\right) \mbox{ for every } 1\leq l\leq s.
\end{equation}

Set $\mathbf{y}_0 := (0)$. Now (\ref{eUB1}) and (\ref{eUB}) imply
\[
| f(\mathbf{x})|\leq |f(\mathbf{y}_0)| +\sum_{i=1}^s |f(\mathbf{y}_0 +\dots +\mathbf{y}_i)-f(\mathbf{y}_0 +\dots +\mathbf{y}_{i-1})| \leq |f(\mathbf{y}_0)|+m.
\]
Thus $f$ is bounded.
\end{proof}
We do not know whether every  real-valued uniformly continuous function on $\mathfrak{F}_0(X)$ can be extended to a  real-valued uniformly continuous function on $(X^\NN,\mathfrak{u})$.

Theorem \ref{t12} is a part of the following one in which we summarize also some known results from \cite{DMPT, Ga8}.

\begin{theorem}
Let $X$ be a  compact connected metrizable Abelian group. Then
\begin{enumerate}
\item[{\rm (i)}] {\rm (\cite{DMPT})} $\mathfrak{F}_0 (X)$ is a connected Polish  Abelian group.
\item[{\rm (ii)}] $\mathfrak{F}_0 (X)$ is   monothetic.
\item[{\rm (iii)}] $\mathfrak{F}_0 (X)$ is  a  non-Schwartz group.
\item[{\rm (iv)}] {\rm (\cite{DMPT, Ga8})} $\mathfrak{F}_0 (X)$ has countable dual. More precisely, $\widehat{\mathfrak{F}_0 (X)}= \widehat{X}^{(\mathbb{N})}$.
\item[{\rm (v)}] {\rm (\cite{Ga8})} $\mathfrak{F}_0 (X)$ is  reflexive.
\item[{\rm (vi)}] $\mathfrak{F}_0 (X)$ does not respect any property from $\mathfrak{P}$. In particular, $\mathfrak{F}_0 (X)$ does not have the Schur property.
\item[{\rm (vii)}] $\mathfrak{F}_0 (X)$ has the property UB.
\item[{\rm (viii)}] $c_0\left( \mathfrak{F}_0 (X)\right)$ is $\mathfrak{g}$-closed in $\mathfrak{F}_0 (X)^\NN$.
\end{enumerate}
\end{theorem}

\begin{proof} (i). The group $\mathfrak{F}_0 (X)$ is connected and Polish by Fact \ref{fTopPropF}.
(ii) follows from Proposition \ref{pMon}.
(iii). The group $\mathfrak{F}_0 (X)$ is not a  Schwartz group by Theorem \ref{t11}.
(vi) follows from Theorem  \ref{t11}.
(vii) follows from Proposition \ref{pUB}.

(viii). To prove that $c_0\left( \mathfrak{F}_0 (X)\right)$ is $\mathfrak{g}$-closed in $\mathfrak{F}_0 (X)^\NN$ it is enough to show that for every $(\mathbf{y}_n)_{n\in\NN} \in \mathfrak{F}_0 (X)^\NN \setminus c_0\left( \mathfrak{F}_0 (X)\right)$ there exists a sequence $\{\chi_l \}_{l\in\NN} \in \widehat{\mathfrak{F}_0 (X)}^{(\mathbb{N})}$ such that
\begin{enumerate}
\item[(1)] $(\chi_l, (\mathbf{x}_n)) \to 1 \mbox{ at } l\to\infty, \quad \forall (\mathbf{x}_n) \in c_0\left( \mathfrak{F}_0 (X)\right),$ and
\item[(2)] $(\chi_l, (\mathbf{y}_n))\not\to 1  \mbox{ at } l\to\infty$.
\end{enumerate}

Let $\rho$ be an invariant metric for $X$ and $\mathbf{y}_n =(y^n_k)_{k\in\NN} \in c_0(X)$. Since $\mathbf{y}_n \not\to 0$ in $\mathfrak{F}_0 (X)$ we can find an increasing sequence $\{n_l\}_{l\in\NN}$ such that
\[
d\left(0, \mathbf{y}_{n_l}\right) = \sup_{k\in\NN} \rho\left( 0, y^{n_l}_k\right)> \delta >0 , \quad \forall l\in\NN.
\]
So, for every $l\in\NN$ we can choose $k_l\in\NN$ such that
\begin{equation} \label{eResp1}
\rho\left( 0, y^{n_l}_{k_l}\right)> \delta.
\end{equation}
Since $X$ is compact and metrizable, without loss of generality we can assume that $y^{n_l}_{k_l}$ converges to an element $y\in X$. It follows from (\ref{eResp1}) that
$\rho\left( 0, y\right)\geq \delta >0.$
In particular, $y\not= 0$. Choose $g\in \widehat{X}$ such that
\begin{equation} \label{eResp3}
( g, y)\not= 1.
\end{equation}
Set
\[
\chi_l = \nu_{n_l} (\mathbf{g}_l) \in \widehat{\mathfrak{F}_0 (X)}^{(\mathbb{N})} , \mbox{ where } \mathbf{g}_l = \nu_{k_l} (g) \in \widehat{\mathfrak{F}_0 (X)}.
\]
Let us check (1) and (2).

Fix  $(\mathbf{x}_n) \in c_0\left( \mathfrak{F}_0 (X)\right)$, where $\mathbf{x}_n =(x^n_k)_{k\in\NN} \in c_0(X)$. This means that
\begin{equation} \label{eResp4}
d\left(0, \mathbf{x}_{n}\right) = \sup_{k\in\NN} \rho\left( 0, x^{n}_k\right) \to 0 \mbox{ at } n\to\infty.
\end{equation}
It follows from (\ref{eResp4}) that $x^{n_l}_{k_l} \to 0 \mbox{ in } X.$
Hence
\[
(\chi_l, (\mathbf{x}_n))=\left( \mathbf{g}_l, \mathbf{x}_{n_l} \right) = \left( g, x^{n_l}_{k_l}\right) \to 1, \mbox{ at } l\to\infty.
\]
This proves (1). The inequality (\ref{eResp3}) yields
\[
(\chi_l, (\mathbf{y}_n))=\left( \mathbf{g}_l, \mathbf{y}_{n_l} \right) = \left( g, y^{n_l}_{k_l}\right) \to (g, y) \not= 1,  \mbox{ at } l\to\infty,
\]
that proves (2). Thus  $c_0\left( \mathfrak{F}_0 (X)\right)$ is $\mathfrak{g}$-closed in $\mathfrak{F}_0 (X)^\NN$.
\end{proof}


\section{The Glicksberg and Schur properties and $k$- and $s$-groups} \label{secCateg}


In this section  we show that the Glicksberg and Schur properties can be naturally defined by two natural functors in $\mathbf{TG}$. For this  we consider two important classes of topological groups introduced by Noble in \cite{Nob, Nob2}, namely, $k$- and $s$-groups.

(I) {\it The Glicksberg property and $k$-groups.} For every $(X,\tau)\in \mathbf{TG}$ denote by $k(\tau)$ the finest group topology for $X$ coinciding on compact sets with $\tau$. In particular, $\tau$ and $ k(\tau)$ have the same family of compact subsets. Clearly, $\tau\leq k(\tau)$. If $\tau = k(\tau)$, the group $(X,\tau)$ is called a {\it $k$-group} \cite{Nob2}. The group $(X,k(\tau))$ is called the {\it $k$-modification} of $X$. The assignment $\mathbf{k}(X,\tau) := (X,k(\tau))$ is a coreflector from $\mathbf{TG}$ to the full subcategory $\mathbf{K}$ of all  $k$-groups. The class $\mathbf{K}$   contains all topological groups whose underlaying space is a $k$-space. In particular, the class $\mathbf{LC}$ of all locally compact groups  is contained in $\mathbf{K}$. Since every metrizable group is a $k$-space we have $\mathbf{LC} \subsetneqq  \mathbf{K}$. The family of all Abelian $k$-groups we denote by $\mathbf{KA}$.

Denote by $\mathbf{PCom}$ the class of all precompact  groups, and by $\mathfrak{RC}$ the class of all MAP groups which respect compactness. The item (1) in the next proposition shows that respect compactness is naturally defined by the functors $\mathbf{k}$ and $\mathfrak{B}$. Note also that item (5) generalizes  \cite[Theorem 1.2]{Tri91} (see also Remark in \cite{BaMP2}).
\begin{proposition} \label{p61}
Let $X$ and $Y$ be MAP topological groups.
\begin{enumerate}
\item[{\rm (1)}] $X\in \mathfrak{RC}$ if and only if $(\mathbf{k}\circ\mathfrak{B})(X) =\mathbf{k}(X)$.
\item[{\rm (2)}] $X\in  \mathbf{K}\cap \mathfrak{RC}$ if and only if $(\mathbf{k}\circ\mathfrak{B})(X) =X$.
\item[{\rm (3)}] $\mathbf{PCom}\subsetneqq  \mathfrak{RC}$ and $\mathbf{LCA} \subsetneqq  \mathbf{KA}\cap \mathfrak{RC}$.
\item[{\rm (4)}] $\mathbf{K}\cap \mathfrak{RC} \subsetneqq \mathbf{K}$ and $\mathbf{K}\cap \mathfrak{RC}\subsetneqq \mathfrak{RC}$.
\item[{\rm (5)}] Let $X\in \mathbf{K}$ and $Y\in \mathfrak{RC}$ and let $\phi: X\to Y$ be a homomorphism. If $\phi^+: X^+ \to Y^+, \phi^+(x):=\phi(x),$ is continuous, then $\phi$ is continuous.
\end{enumerate}
\end{proposition}

\begin{proof}
(1). If $X\in \mathfrak{RC}$, then $(\mathbf{k}\circ\mathfrak{B})(X) =\mathbf{k}(X)$ by the definition of respect compactness  and the definition of $\mathbf{k}(X)$.

Conversely, let $(\mathbf{k}\circ\mathfrak{B})(X) =\mathbf{k}(X)$ and $K$ is compact in $\mathfrak{B}(X)$. Then $K$ is compact in $(\mathbf{k}\circ\mathfrak{B})(X)$ by the definition of $k$-modification. So $K$ is compact in $\mathbf{k}(X)$. Hence,  by the definition of $k$-modification,  $K$ is compact  in $X$. Thus $X\in \mathfrak{RC}$.


(2). Let $X\in  \mathbf{K}\cap \mathfrak{RC}$. By item (1) and the definition of $k$-groups, we obtain $(\mathbf{k}\circ\mathfrak{B})(X) =\mathbf{k}(X)=X$.

Conversely, let $(\mathbf{k}\circ\mathfrak{B})(X) =X$. Since $\mathbf{k}\circ\mathbf{k}=\mathbf{k}$, the equalities
\[
\mathbf{k}(X)=\mathbf{k}\circ(\mathbf{k}\circ \mathfrak{B}(X)) =(\mathbf{k}\circ \mathfrak{B})(X) =X
\]
and item (1) imply that $X$ is a $k$-group and $X\in \mathfrak{RC}$.

(3). Since $\mathfrak{B}(K)=K$ for each precompact group $K$, the first inclusion follows. The second one holds by the Glicksberg theorem. To prove that these  inclusions are strict take an arbitrary compact totally disconnected metrizable group $X$. Then $\mathfrak{F}_0(X)$ is metrizable, and hence it is a $k$-group. Now Theorem \ref{t11} and Fact \ref{fTopPropF}(iii) imply that $\mathfrak{F}_0(X)$ respects compactness and it is not locally precompact. Thus the inclusions  are strict.

(4). Being metrizable the group $\mathfrak{F}_0(\TT)$ belongs to $\mathbf{KA}$. However, $\mathfrak{F}_0(\TT)$ does not respect compactness by Theorem \ref{t11}. Thus $\mathbf{K}\cap \mathfrak{RC} \not= \mathbf{K}$.

To prove that the second inclusion is strict it is enough to find a precompact Abelian group $X$ which is not a $k$-group.
Take an arbitrary non-measurable subgroup $H$ of the circle $\TT$ and set $X:=(\ZZ, T_H)$. Then  the precompact group $X$ does not contain non-trivial convergent sequences  (see \cite{CRT}). Since $X$ is countable, we obtain that $X$ also has no infinite  compact subsets by  \cite[3.1.21]{Eng}. This immediately implies that the $k$-modification $\mathbf{k}(X)$ of $X$ is discrete. Hence $\mathbf{k}(X)=\ZZ_d$ is a discrete LCA group. So $\mathbf{k}(X)\not= X$ and $X$ is not a $k$-group. Thus the second inclusion is  strict.

(5). Let $id_X: X\to X^+$ and $id_Y: Y\to Y^+$ be the identity continuous maps. Fix arbitrarily   a compact subset $K$ in $X$. Then $K^+ :=\phi^+(id_X(K))$ is compact in $Y^+$. As $Y\in \mathfrak{RC}$, $K^+$ is compact in $Y$. So $id_Y|_{K^+}$ is a homeomorphism.   Hence $\phi|_K = (id_Y|_{K^+})^{-1} \circ \phi^+\circ (id_X |_K)$ is continuous. So $\phi$ is $k$-continuous. As $X$ is a $k$-group, $\phi$ is continuous (see \cite{Nob2}).
\end{proof}

\begin{remark} {\em
In \cite{CTW}, the authors show that the answer to both questions posed in \cite[1.2]{CTW}  (see also \cite{Tri91}) is ``no''. Let us show that the group $X$ in the proof of item (4) of Proposition \ref{p61} also  answers negatively  to those questions. We use the notation from  \cite[1.2]{CTW}. Set $G=\mathbb{Z}$ and  $\mathcal{U}=T_H$. Since $G$ is countable, every locally compact group topology $\mathcal{T}$ on $G$ must be discrete. So $\mathcal{T}^+ = T_\mathbb{T}$. Further, as it was noticed in item (4), a subset $A$ of $G$ is $\mathcal{T}$-compact if and only if $A$ is $\mathcal{U}$-compact (if and only if $A$ is finite). However, since $H \not= \TT$, we obtain $\mathcal{U} \not= \mathcal{T}^+$ by Fact \ref{fBohr}. }
\end{remark}


We do not know an answer to the next questions:
\begin{problem} \label{problem5}
{\it Let $X\in \mathbf{KA}$ (or just a $k$-space). Is  $\mathfrak{F}_0(X)$ a $k$-group}?
\end{problem}

\begin{problem} 
{\it Let $X$ be a LCA group. Is  $\mathfrak{F}_0(X)$ a $k$-space}?
\end{problem}

(II) {\it The Schur property and $s$-groups.}
Similar to $k$-groups  we define $s$-groups (we follow \cite{Ga30}).
Let $(X,\tau)$ be a (Hausdorff) topological group and let $S$ be the set of all sequences in $(X,\tau)$ converging to the unit. Then there exists the finest Hausdorff group topology $\tau_S$ on the underlying group $X$ in which all sequences of $S$ converge to the unit.  If $\tau =\tau_S$, the group $X$ is called an {\it $s$-group}. The assignment $\mathbf{s}(X,\tau) := (X,\tau_S)$ is a coreflector from $\mathbf{TG}$ to the full subcategory $\mathbf{S}$ of all  $s$-groups. The class $\mathbf{S}$   contains all sequential groups \cite[1.14]{Ga30}. Every $s$-group is also a $k$-group \cite{Ga3}.  Note that $X$ and $\mathbf{s}(X)$ have the same set of convergent sequences \cite[4.2]{Ga30}. The family of all Abelian $s$-groups we denote by $\mathbf{SA}$.

Denote by  $\mathfrak{RS}$ the class of all MAP  groups which have the Schur property. In analogy to Proposition \ref{p61} we obtain:
\begin{proposition} \label{p62}
Let $X$ and $Y$ be MAP topological groups.
\begin{enumerate}
\item[{\rm (1)}] $X\in \mathfrak{RS}$ if and only if $(\mathbf{s}\circ\mathfrak{B})(X) =\mathbf{s}(X)$.
\item[{\rm (2)}] $X\in  \mathbf{S}\cap \mathfrak{RS}$ if and only if $(\mathbf{s}\circ\mathfrak{B})(X) =X$.
\item[{\rm (3)}] $\mathbf{S}\cap \mathfrak{RS} \subsetneqq \mathbf{S}$ and $\mathbf{S}\cap \mathfrak{RS}\subsetneqq \mathfrak{RS}$.
\item[{\rm (4)}] Let $X\in \mathbf{S}$ and $Y\in \mathfrak{RS}$ and let $\phi: X\to Y$ be a homomorphism. If $\phi^+: X^+ \to Y^+, \phi^+(x):=\phi(x),$ is continuous, then $\phi$ is continuous.
\end{enumerate}
\end{proposition}

\begin{proof}
(1). If $X\in \mathfrak{RS}$, then $(\mathbf{s}\circ\mathfrak{B})(X) =\mathbf{s}(X)$ by the definition of respect sequentiality  and the definition of $\mathbf{s}(X)$.

Conversely, let $(\mathbf{s}\circ\mathfrak{B})(X) =\mathbf{s}(X)$ and $x_n \to e$ in $X^+$. Then $x_n \to e$ in $(\mathbf{s}\circ\mathfrak{B})(X)$ by \cite[Lemma 4.2]{Ga30}, and hence $x_n \to e$ in $\mathbf{s}(X)$. So $x_n \to e$ in $X$ by \cite[Lemma 4.2]{Ga30}. Thus $X\in \mathfrak{RS}$.

(2). Let $X\in  \mathbf{S}\cap \mathfrak{RS}$. By item (1) and the definition of $s$-groups, we obtain $(\mathbf{s}\circ\mathfrak{B})(X) =\mathbf{s}(X)=X$.

Conversely, let $(\mathbf{s}\circ\mathfrak{B})(X) =X$. Since $\mathbf{s}\circ\mathbf{s}=\mathbf{s}$, the equalities
\[
\mathbf{s}(X)=\mathbf{s}\circ(\mathbf{s}\circ \mathfrak{B}(X)) =(\mathbf{s}\circ \mathfrak{B})(X) =X
\]
and item (1) imply that $X$ is an $s$-group and $X\in \mathfrak{RC}$.

(3). Being metrizable the group $\mathfrak{F}_0(\TT)$ belongs to $\mathbf{SA}$. However, $\mathfrak{F}_0(\TT)$ does not respect sequentiality by Theorem \ref{t11}. Thus $\mathbf{S}\cap \mathfrak{RS} \not= \mathbf{S}$.
The example in the proof of Proposition \ref{p61}(4) shows also that the second inclusion is strict.

(4).  Let $id_X: X\to X^+$ and $id_Y: Y\to Y^+$ be the identity continuous maps. Fix arbitrarily   a  convergent sequence $\uuu$  with the limit point in $X$ (so $\uuu$ is a compact subset of $X$). Then $\uuu^+ :=\phi^+(id_X(\uuu))$ is a  convergent sequence with the limit point in $Y^+$. As $Y\in \mathfrak{RS}$,  $\uuu^+$ is a  convergent sequence in $Y$. So $id_Y|_{\uuu^+}$ is a homeomorphism.  Hence $\phi|_\uuu = (id_Y|_{\uuu^+})^{-1} \circ \phi^+\circ (id_X |_\uuu)$ is continuous. So $\phi$ is a sequentially continuous homomorphism. As $X$ is an $s$-group, $\phi$ is continuous  \cite{Ga30}.
\end{proof}
The next question is open:
\begin{problem} \label{problem6}
{\it Let $X\in \mathbf{SA}$ (in particular, the group $X$ is sequential). Is  $\mathfrak{F}_0(X)$ an $s$-group}?
\end{problem}

\section{Open questions}

In this section we define another functors naturally coming from Functional Analysis and pose some open questions.

(I)  {\it  Functors $\mathfrak{F}_0^I$, $\mathfrak{F}_\infty^I$ and $\mathfrak{F}^I$ on $\mathbf{TAG}$}. Let $I$ be an arbitrary infinite set of indices and $X$ be an Abelian topological group. The collection $\{ U^I : U \in \mathcal{N}(X)\}$ forms a base at $0$ for a group topology $\mathfrak{u}$ in $X^I$. We   call $\mathfrak{u}$  the {\it uniform topology}.

We denote by $c_0^I(X)$ the set of all $(x_i)_{i\in I} \in X^I$ such that, for every $U\in \mathcal{N}(X)$, $x_i \in U$ for all but finitely many indices.
The set of all $(x_i)_{i\in I} \in X^I$ such that the set $\{ x_i\}_{i\in I}$ is precompact in $X$, we denote by $\ell_\infty^I (X)$.

The uniform group topology on $c_0^I (X)$ and $\ell_\infty^I (X)$ induced from $(X^I, \mathfrak{u})$ we denote by  $\mathfrak{u}_0$ and $\mathfrak{u}_\infty$ respectively. In analogy to $\mathfrak{F}_0$ we define the functors $\mathfrak{F}_0^I$, $\mathfrak{F}_\infty^I$ and $\mathfrak{F}^I$ on  the category $\mathbf{TAG}$ by the assignment
\[
\begin{split}
X & \to \mathfrak{F}^I (X) :=(X^I , \mathfrak{u}),\\
X & \to \mathfrak{F}_0^I (X) :=(c_0^I (X), \mathfrak{u}_0),\\
X & \to \mathfrak{F}_\infty^I (X) :=(\ell_\infty^I (X), \mathfrak{u}_\infty).
\end{split}
\]
In the case $I=\NN$ we shall omit the subscript $I$.
If $X=\mathbb{R}$, then $\mathfrak{F}_\infty (\mathbb{R})$ coincides with the classical Banach space $\ell^\infty$. On the other hand, the group $\mathfrak{F} (\mathbb{R})$ is not a TVS. Clearly, $\mathfrak{F}^I$ and $\mathfrak{F}_\infty^I $ coincide on precompact groups.

It would be interesting to consider Problems \ref{problem1}-\ref{problem4}  for these functors. Let us note that it is not clear even whether $\mathfrak{F}_\infty^I (X)$ is connected for a compact connected group $X$, although it is known that $\mathfrak{F}_\infty (\TT)$ is even arc-connected.
Clearly, to obtain an analogue to Fact \ref{f1} and Theorem \ref{t11} we should describe the dual group of  $\mathfrak{F}_\infty^I (X)$.
\begin{problem}
{\it Let $X$ be a (compact) LCA group. Describe $\widehat{\mathfrak{F}_\infty^I (X)}$ or $\widehat{\mathfrak{F}_\infty (X)}$}.
\end{problem}
Note that these groups are very ``big''. For example, it is known that $\mathrm{Card}\left( \widehat{\mathfrak{F}_\infty (\TT)}\right) = 2^\mathfrak{c}$ (see \cite{DMPT13}). However, even for  the simplest case $X=\TT$ we know only the cardinality of the dual group $\widehat{\mathfrak{F}_\infty (\TT)}$, but not its appropriate description (as, for example, in Fact \ref{f3}).

The following question is of interest:
\begin{problem}
{\it Let $X$ be a (compact) LCA group and $I$ be an infinite set. Are the groups $\mathfrak{F}_0^I (X)$,  $\mathfrak{F}_\infty^I (X)$ and $\mathfrak{F}^I (X)$ reflexive}?
\end{problem}
By Fact \ref{f1} we know only that $\mathfrak{F}_0 (X)$ is reflexive.

(II)
Let $\mathcal{P}_1, \mathcal{P}_2 \in \mathfrak{P}$ be distinct properties and let $\Gamma$ be an arbitrary family of Abelian topological groups. We say that $\mathcal{P}_1 \leq_R \mathcal{P}_2$ on $\Gamma$ if every $X\in \Gamma$ which respects $\mathcal{P}_1$ respects also $\mathcal{P}_2$. If $\mathcal{P}_1 \leq_R \mathcal{P}_2$ and $\mathcal{P}_2 \leq_R \mathcal{P}_1$ on $\Gamma$ we say that $\mathcal{P}_1$ and $\mathcal{P}_2$ are {\it equivalent} on $\Gamma$. So, all properties from $\mathfrak{P}$ are equivalent on $\mathbf{LCA}$ by \cite{Tri91}, and all properties from $\mathfrak{P}_0$ are equivalent on the class of all nuclear groups by \cite{BaMP2}. Further, the properties $\mathcal{C}, \mathcal{CC}, \mathcal{PC}$ and $\mathcal{FB}$ are equivalent on the class of all complete Abelian $g$-groups by \cite[Theorem 3.3]{HM}. These results justify the next question (cf. \cite[Question 4.6]{Aus2}):
\begin{problem}
{\it Are the properties from $\mathfrak{P}$ or $\mathfrak{P}_0$ equivalent on the class of all (complete) MAP locally quasi-convex Schwartz groups}?
\end{problem}
Theorem \ref{t-Res} says that the answer to this question is ``yes'' if the next question has the positive answer:
\begin{problem}
{\it Let $X$ be a (complete) MAP locally quasi-convex Schwartz groups. Does $X$ have the $SB$-property}?
\end{problem}

As it was mentioned, there are  Schur groups  without the Glicksberg property \cite{Wil} (see a complete proof in \cite[Example 19.19]{Dom}).   Taking into account also Diagram \ref{diag} it seems to be interesting the next question.
\begin{problem}
{\it Let $\mathcal{P}_1, \mathcal{P}_2 \in \mathfrak{P}$ be distinct  properties and $\mathcal{P}_1 \leq_R \mathcal{P}_2$ on $\mathbf{MAPA}$. Find a MAP Abelian group $X$ which respects $\mathcal{P}_2$ but which does not respect $\mathcal{P}_1$.}
\end{problem}

It would be interesting also to consider Problems \ref{problem4},  \ref{problem5} and \ref{problem6} for other classes of MAP Abelian groups (as, for instance, Polish, nuclear or Schwartz MAP groups).

In connection with  Fact \ref{fGclosed} and Theorem \ref{t12}(iii) the next problem is of interest.
\begin{problem}
{\it Let $X$ be an Abelian topological group such that $c_0(X)$ is $\mathfrak{g}$-closed in $X^\NN$. Under which conditions the group $X$ has the Schur property?}
\end{problem}

\vspace{3mm}

{\bf Acknowledgments}.
I wish to thank  Professor D.~Dikranjan for the comment \cite{Dikr} and suggestions.  It is a pleasure to thank Professors S.~Hern\'{a}ndez, E.~Mart\'{\i}n-Peinador and V.~Tarieladze for pointing out the articles \cite{Her98, DMPT, MT}.

\end{document}